\documentclass[onefignum,onetabnum]{siamart250211}

\usepackage{lipsum}
\usepackage{amsfonts}
\usepackage{amsmath}
\usepackage{graphicx}
\usepackage{epstopdf}
\usepackage{algorithmic}
\usepackage{bm}
\usepackage{tikz}
\ifpdf
  \DeclareGraphicsExtensions{.eps,.pdf,.png,.jpg}
\else
  \DeclareGraphicsExtensions{.eps}
\fi

\newsiamremark{remark}{Remark}
\newsiamremark{hypothesis}{Hypothesis}
\crefname{hypothesis}{Hypothesis}{Hypotheses}
\newsiamthm{claim}{Claim}
\newsiamremark{fact}{Fact}
\crefname{fact}{Fact}{Facts}

\newcommand{\bu}{{\bf u}}
\newcommand{\dx}{{\rm d}x}
\newcommand{\dy}{{\rm d}y}
\newcommand{\dz}{{\rm d}z}
\newcommand{\dt}{{\rm d}t}
\newcommand{\dzero}{{\rm d}_0}
\newcommand{\done}{{\rm d}_1}
\newcommand{\dtwo}{{\rm d}_2}
\newcommand{\dthree}{{\rm d}_3}

\newcommand{\trinorm}[1]{{\vert\kern-0.25ex\vert\kern-0.25ex\vert #1 \vert\kern-0.25ex\vert\kern-0.25ex\vert}}

\headers{Unified Spatiotemporal Convection-Diffusion}{
J. H. Adler, X. Hu, and S. Lee}

\title{A Unified Spatiotemporal Formulation with Physics-Preserving Structure for Time-Dependent Convection-Diffusion Problems\thanks{Submitted to the editors DATE.
\funding{The authors received no external funding for this work.}}}

\author{James H. Adler\thanks{Department of Mathematics, Tufts University, Medford, MA 02155  (\email{james.adler@tufts.edu}, \email{xiaozhe.hu@tufts.edu}, \email{seulip.lee@tufts.edu}).}
\and Xiaozhe Hu\footnotemark[2]
\and Seulip Lee\footnotemark[2]}

\usepackage{amsopn}

\ifpdf
\hypersetup{
  pdftitle={A Unified Formulation for Time-Dependent Convection-Diffusion},
  pdfauthor={J. H. Adler, X. Hu, and S. Lee}
}
\fi

\begin{document}

\maketitle

\begin{abstract}
We propose a unified four-dimensional (4D) spatiotemporal formulation for time-dependent convection-diffusion problems that preserves underlying physical structures. By treating time as an additional space-like coordinate, the evolution problem is reformulated as a stationary convection-diffusion equation on a 4D space-time domain. Using exterior calculus, we extend this framework to the full family of convection-diffusion problems posed on $H(\textbf{grad})$, $H(\textbf{curl})$, and $H(\textnormal{div})$. The resulting formulation is based on a 4D Hodge-Laplacian operator with a spatiotemporal diffusion tensor and convection field, augmented by a small temporal perturbation to ensure nondegeneracy. This formulation naturally incorporates fundamental physical constraints, including divergence-free and curl-free conditions. We further introduce an exponentially-fitted 4D spatiotemporal flux operator that symmetrizes the convection-diffusion operator and enables a well-posed variational formulation. Finally, we prove that the temporally-perturbed formulation converges to the original time-dependent convection-diffusion model as the perturbation parameter tends to zero.

\end{abstract}



\section{Introduction}
    Time-dependent convection-diffusion problems describe the transport of physical quantities such as heat, mass, and momentum in diverse applications, from porous-media flow and thermal diffusion to atmospheric, oceanic, and plasma dynamics. In many applications, the transport is convection-dominated, meaning that the advective effects greatly exceed diffusive dissipation.
    In such settings, solutions often develop sharp interior or boundary layers aligned with the convection field, which must be carefully resolved. However, from a numerical perspective, resolving these features is challenging: standard discretizations may violate the discrete maximum principle \cite{burman2005stabilized,ciarlet1973maximum,roos2008robust}, resulting in spurious oscillations \cite{elman2014finite,john2007spurious}, while stabilized methods designed to suppress such oscillations (e.g., \cite{bacuta2025convergence,barrenechea2018unified,brooks1982streamline,knobloch2010generalization}) often introduce excessive numerical dissipation, leading to artificial smearing and loss of physical resolution. We refer to \cite{barrenechea2024finite,kikuchi1977discrete} for further discussion of discrete maximum principles and artificial diffusion. Therefore, balancing stability, accuracy, and the preservation of essential physical properties in convection-dominated transport remains a central challenge in numerical simulation.

    A classical difficulty in the numerical treatment of time-dependent convection-diffusion problems arises from the separate treatment of spatial and temporal discretizations, often via time-stepping or operator-splitting methods. These approaches intertwine stability conditions across time and space and can accumulate numerical dissipation as the solution evolves. An alternative \cite{bank2017arbitrary} is to regard time as an additional spatial coordinate and reformulate the problem on the $(n+1)$-dimensional space-time domain $\Omega = \Omega_{\mathbf{x}} \times \Omega_{t}$. In this spatiotemporal setting, the original evolution equation becomes a stationary convection-diffusion problem on a higher-dimensional domain, with the temporal dependence incorporated through a spatiotemporal diffusion tensor and convection field. Although this reformulation increases the dimensionality of the computational domain, it removes the need for time-stepping and places spatial and temporal variations on equal footing. This is particularly advantageous in convection-dominated regimes: stability and accuracy can be enforced at the global space-time level while still controlling numerical dissipation, rather than amplifying it through successive temporal updates. Moreover, the formulation can integrate naturally with adaptive refinement strategies and existing solver libraries for stationary problems, providing an efficient computational framework that accurately resolves sharp transport layers.

    In this work, we develop a unified spatiotemporal formulation with physics-preserving structure for time-dependent convection-diffusion problems using the language of differential forms and exterior calculus. Our goal is to extend the spatiotemporal approach in \cite{bank2017arbitrary} to the full family of time-dependent convection-diffusion problems. While the space-time formulation in \cite{bank2017arbitrary} provides a spatiotemporal representation for the $H(\textbf{grad})$ problem, a direct extension to the $H(\textbf{curl})$ and $H(\mathrm{div})$ settings is not available within classical vector calculus, due to the absence of natural 4D analogues of the curl operator and vector field structures. Although several works \cite{heumann2010eulerian,hiptmair2015discretizing,wu2020simplex} have employed differential forms to study $H(\textbf{curl})$ and $H(\textnormal{div})$ convection-diffusion problems, these approaches either treat space and time separately or remain restricted to 3D steady-state settings. To overcome this difficulty, we formulate all operators within a 4D exterior-calculus framework by introducing a spatiotemporal diffusion tensor and convection field, augmented by a small artificial temporal perturbation proposed in \cite{bank2017arbitrary} to ensure the nondegeneracy of the diffusion tensor. We further define unified boundary conditions suitable for the spatiotemporal setting that incorporate both Dirichlet and Neumann types. Remarkably, the unified 4D spatiotemporal governing equation naturally embeds physics-preserving structures, including Gauss's law for magnetism (divergence-free condition) and a curl-free condition.

    In addition, we introduce an exponentially-fitted spatiotemporal flux operator, inspired by semiconductor simulations \cite{brezzi1989numerical,brezzi1989two,miller1991triangular}, and demonstrate its applicability to our unified framework. This operator symmetrizes fluxes in convection-diffusion problems, transferring the structural properties of a Hodge Laplacian to the convection-diffusion setting and providing the key analytical foundation for developing future physics-preserving 4D discretizations in convection-dominated regimes (see, e.g., \cite{adler2023stable,lazarov2012exponential,wu2020simplex} for related 2D and 3D approaches). Building on the flux operator, we formulate the variational problem for the unified formulation and establish its well-posedness. Finally, we show that the temporally-perturbed spatiotemporal formulation converges to the original convection-diffusion model as the perturbation parameter tends to zero, ensuring consistency with the standard formulation.

    The remainder of the paper is organized as follows. In \Cref{sec: prelim}, we review the time-dependent convection-diffusion problems, reintroduce the spatiotemporal formulation for the $H(\textbf{grad})$ case as in \cite{bank2017arbitrary}, and define the spatiotemporal exterior calculus framework, which serves as the foundation for the unified equation. Building on this groundwork, \Cref{sec: unified_spatiotemporal} presents the unified spatiotemporal convection-diffusion equation, specifies its structure at each differential-form level, and discusses the associated boundary conditions. In \Cref{sec: well-posed}, we develop the exponentially-fitted flux operator in the spatiotemporal setting, formulate the corresponding variational problem, and establish well-posedness of the unified formulation. \Cref{sec: convergence} provides a convergence analysis of the temporally-perturbed formulation, showing that it approaches the original convection-diffusion model as the perturbation parameter tends to zero. Finally, \Cref{sec: conclusions} summarizes the main findings and outlines directions for future research.

\section{Preliminaries}\label{sec: prelim}
In this section, we introduce the time-dependent convection-diffusion problems in the $H(\textbf{grad})$, $H(\textbf{curl})$, $H(\textnormal{div})$, and purely temporal settings. We then review the spatiotemporal formulation for the $H(\textbf{grad})$ case as in \cite{bank2017arbitrary}, and use it as motivation to develop the spatiotemporal exterior calculus framework.

\subsection{Convection-diffusion problems}
In an open, bounded spatial domain $\Omega_{\mathbf{x}}\subset\mathbb{R}^3$ and over the time interval $\Omega_t=(t_0,T)$, we establish the classical time-dependent $H(\textbf{grad})$ convection-diffusion equation, which models heat conduction, mass transport, and contaminant dispersion in fluids (see e.g. \cite{donea2003finite,hughes2012finite,zheng2002applied}): \begin{subequations}\label{sys: conv_diff_grad}
\begin{alignat}{2}
\frac{\partial {u}}{\partial t} - \nabla\cdot(\alpha\nabla{u} + \bm{\beta}u)&=f &&  \quad\text{in }\Omega_\mathbf{x}\times\Omega_t,\label{eqn: grad_equation}\\
u &= 0 && \quad\text{on }\partial\Omega_\mathbf{x},\label{eqn: grad_boundary}\\
u(\mathbf{x},t_0)&=u_{t_0}(\mathbf{x})&&\quad\text{in } \Omega_\mathbf{x}.\label{eqn: grad_initial}
\end{alignat}
\end{subequations}
Here, $u$ denotes the transported quantity with initial condition $u_{t_0}(\mathbf{x})$, $\alpha=\alpha(\mathbf{x},t)>0$ is the diffusion coefficient, $\bm{\beta}=\bm{\beta}(\mathbf{x},t)$ is the convection (advection) field, and $f=f(\mathbf{x},t)$ represents a source or forcing term. The operator $\nabla=\langle\partial_x,\partial_y,\partial_z\rangle$ acts only on spatial variables, where $\mathbf{x} = \langle x,y,z\rangle$.

Similarly, problem \eqref{sys: conv_diff_curl} represents a time-dependent $H(\textbf{curl})$ convection-diffusion equation involving curl operators. Such formulations arise in models of electromagnetic induction~\cite{krause1976mean}, Maxwell's equations in conducting media~\cite{vujevic2024new}, and magnetohydrodynamics (MHD)~\cite{davidson2017introduction,galtier2016introduction}: \begin{subequations}\label{sys: conv_diff_curl}
\begin{alignat}{2}
\frac{\partial \mathbf{u}}{\partial t} + \nabla\times(\alpha\nabla\times \mathbf{u} + \bm{\beta}\times\mathbf{u})&=\mathbf{f} &&  \quad\text{in }\Omega_\mathbf{x}\times\Omega_t,\label{eqn: curl_equation}\\
\mathbf{n}\times\mathbf{u} &= \mathbf{0} && \quad\text{on }\partial\Omega_\mathbf{x},\label{eqn: curl_boundary}\\
\mathbf{u}(\mathbf{x},t_0)&=\mathbf{u}_{t_0}(\mathbf{x})&&\quad\text{in } \Omega_\mathbf{x},\label{eqn: curl_initial}
\end{alignat}
\end{subequations}
where $\mathbf{u}$ denotes the vector-valued field (e.g., magnetic or electric field) with initial condition $\mathbf{u}_{t_0}(\mathbf{x})$, $\mathbf{f}=\mathbf{f}(\mathbf{x},t)$ is a vector-valued source or forcing term, 
and $\mathbf{n}$ denotes the unit outward normal vector on $\partial\Omega_\mathbf{x}$.
Next, problem \eqref{sys: conv_diff_div} represents the time-dependent $H(\textnormal{div})$ convection-diffusion equation, which naturally arises from the 2-form convection-diffusion equation~\cite{heumann2010eulerian,wu2020simplex}:\begin{subequations}\label{sys: conv_diff_div}
\begin{alignat}{2}
\frac{\partial \mathbf{u}}{\partial t} - \nabla(\alpha\nabla\cdot\mathbf{u} + \bm{\beta}\cdot\mathbf{u})&=\mathbf{f} &&  \quad\text{in }\Omega_\mathbf{x}\times\Omega_t,\label{eqn: div_equation}\\
\mathbf{u}\cdot\mathbf{n} &= 0 && \quad\text{on }\partial\Omega_\mathbf{x},\label{eqn: div_boundary}\\
\mathbf{u}(\mathbf{x},t_0)&=\mathbf{u}_{t_0}(\mathbf{x})&&\quad\text{in } \Omega_\mathbf{x}.\label{eqn: div_initial}
\end{alignat}
\end{subequations}
Finally, problem \eqref{sys: conv_diff_temporal} corresponds to a purely temporal differential equation parameterized by the spatial coordinate,
\begin{subequations}\label{sys: conv_diff_temporal}
\begin{alignat}{2}
\frac{\partial u}{\partial t}&=f &&  \quad\text{in }\Omega_\mathbf{x}\times\Omega_t,\label{eqn: temporal_equation}\\
u(\mathbf{x},t_0)&=u_{t_0}(\mathbf{x})&&\quad\text{in } \Omega_\mathbf{x},\label{eqn: temporal_initial}
\end{alignat}
\end{subequations}
where each spatial point evolves independently in time under the source term $f(\mathbf{x},t)$.

Our objective is to unify the transient convection-diffusion problems 
\eqref{sys: conv_diff_grad}--\eqref{sys: conv_diff_temporal} within a single stationary mathematical framework. Although each problem is expressed in a distinct functional 
setting, $H(\textbf{grad})$, $H(\textbf{curl})$, $H(\textnormal{div})$, and $L^2$-temporal, they share common physical principles of transport and diffusion. 
This approach not only highlights their structural similarities but also facilitates the development of unified space-time analytical and numerical techniques.

\subsection{Spatiotemporal approach for \texorpdfstring{$H(\textbf{grad})$}{H(grad)}}\label{subsec: Spatiotemporal_Hgrad}

Following the space-time formulation presented in \cite{bank2017arbitrary}, we define the space-time domain as \(\Omega := \Omega_{\mathbf{x}} \times \Omega_t\), and introduce the combined variable \(\mathbf{y} = (\mathbf{x}, t)\).
The boundary of the domain is decomposed as $\partial\Omega=\Gamma_\mathbf{x}\cup\Gamma_{t_0}\cup\Gamma_T$, where
\begin{equation*}
    \Gamma_\mathbf{x} : = \partial\Omega_\mathbf{x}\times\Omega_t,\quad \Gamma_{t_0}: = \Omega_\mathbf{x}\times\{t=t_0\},\quad \Gamma_T := \Omega_\mathbf{x}\times\{t=T\}.
\end{equation*}
We also set $\Gamma_D:=\Gamma_\mathbf{x}\cup\Gamma_{t_0}$ to denote Dirichlet-type boundary conditions, where data are prescribed on the spatial boundary and at the initial time.

Under this setting, the original $H(\textbf{grad})$ time-dependent problem \eqref{sys: conv_diff_grad} can be reformulated as a stationary problem on the spatiotemporal domain \(\Omega \subset \mathbb{R}^{3+1}\). We seek a solution \(u(\mathbf{y}): \Omega \to \mathbb{R}\) such that
\begin{subequations}\label{sys: spatiotemporal-grad}
\begin{alignat}{2}
 -\nabla_{\mathbf{y}} \cdot (D \nabla_{\mathbf{y}} u + \mathbf{b} u) &= f &&\quad \text{in } \Omega,\label{eqn: spatiotemporal-grad-equation}\\
 u &= g &&\quad \text{on } \Gamma_D,\label{eqn: spatiotemporal-grad-Dirichlet}\\
  D\nabla_\mathbf{y} u \cdot \mathbf{n}_{T} &= 0 &&\quad \text{on } \Gamma_{T}.\label{eqn: spatiotemporal-grad-Neumann}
\end{alignat}
\end{subequations}
Here, \(\nabla_{\mathbf{y}}=\langle\partial_x,\partial_y,\partial_z,\partial_t\rangle\) denotes the gradient operator with respect to the spatiotemporal variable \(\mathbf{y} = (\mathbf{x}, t)\). The boundary data \(g\) is determined by the conditions \eqref{eqn: grad_boundary} and \eqref{eqn: grad_initial}, and \(\mathbf{n}_T=\langle\mathbf{0},1\rangle\in \mathbb{R}^{3+1}\) is the outward unit normal vector on \(\Gamma_{T}\).

The spatiotemporal diffusion tensor and convection vector are given by
\[
D = \begin{bmatrix}
\alpha I_3 & \mathbf{0} \\
\mathbf{0}^T & 0
\end{bmatrix}, \quad
\mathbf{b} = \begin{bmatrix}
\bm{\beta} \\
-1
\end{bmatrix},
\]
where $I_3$ denotes the $3\times 3$ identity matrix.
Since \(D\) is degenerate in the temporal direction, we introduce a small artificial second-order time derivative term \(-\varepsilon\, \partial^2 u / \partial t^2\) into the $H(
\textbf{grad}
)$ equation \eqref{eqn: grad_equation}, modifying \(D\) to a nonsingular form (following \cite{bank2017arbitrary}):
\[
D = \begin{bmatrix}
\alpha I_3 & \mathbf{0} \\
\mathbf{0}^T & \varepsilon
\end{bmatrix}.
\]
This modification leads to a convection-dominated regime in the spatiotemporal problem \eqref{sys: spatiotemporal-grad}, characterized by \(\varepsilon\ll1\). 
Moreover, depending on the magnitudes of \(\alpha\) and \(\varepsilon\), the spatiotemporal problem incorporates an anisotropic diffusion tensor.

The Neumann-type boundary condition \eqref{eqn: spatiotemporal-grad-Neumann} evaluates to
\begin{equation}\label{eqn: grad_neumann}
    D\nabla_\mathbf{y} u \cdot \mathbf{n}_{T} = \varepsilon u_t(\mathbf{x},T) = 0,
\end{equation}
which enforces vanishing terminal velocity in time, induced by the artificial temporal perturbation.
Additionally, the flux can be defined as
\begin{equation}
    J(u) := D \nabla_{\mathbf{y}} u + \mathbf{b} u ,
    \label{eqn: flux}
\end{equation}
so the main equation \eqref{eqn: spatiotemporal-grad-equation} can be written compactly as
\begin{equation*}
    -\nabla_{\mathbf{y}} \cdot J(u) = f.
\end{equation*}

The spatiotemporal reformulation offers several advantages, particularly from a numerical standpoint. It enables the use of well-established tools and techniques for stationary problems, including adaptive methods and high-performance linear solvers.
Most notably, when $\alpha$ is very small, standard separate space-time discretizations often fail to capture shocks \cite{john2008finite,kuzmin2004high}: they may introduce spurious oscillations due to the lack of a discrete maximum principle, or they may oversmooth the solution due to excessive numerical dissipation.
This oversmoothing typically becomes more severe over time, as each time-stepping iteration introduces accumulated numerical dissipation.
In contrast, stable numerical methods for convection-dominated regimes based on the stationary formulation have proven effective in resolving boundary layers while preserving the discrete maximum principle \cite{cao2025edge,xu1999monotone}.
This property is particularly valuable for time-dependent transport simulations (convection-dominated), where accurate shock resolution with minimal numerical dissipation is essential.

However, extending the spatiotemporal approach to the $H(\textbf{curl})$ and $H(\text{div})$ convection-diffusion problems presents significant challenges. For the $H(\textbf{curl})$ problem \eqref{sys: conv_diff_curl}, the spatiotemporal formulation requires a 4D curl, which is not standard in classical vector calculus. For the $H(\text{div})$ problem \eqref{sys: conv_diff_div}, a straightforward extension of the spatial vector field to a 4D field of the form $\langle \bu,0\rangle$ fails to reproduce the equation \eqref{eqn: div_equation}, since 
$$\nabla_\mathbf{y}(\nabla_\mathbf{y}\cdot \langle \mathbf{u},0\rangle + \mathbf{b}\cdot\langle\mathbf{u},0\rangle) = \nabla_\mathbf{y}(\nabla\cdot\bu+\bm{\beta}\cdot\bu)=
\begin{bmatrix}
    \nabla(\nabla\cdot\bu+\bm{\beta}\cdot\bu)\\
    (\nabla\cdot\bu+\bm{\beta}\cdot\bu)_t
\end{bmatrix}.$$

To develop a more general and flexible mathematical framework for convection-diffusion problems,  we adopt the language of differential forms and exterior calculus, which naturally accommodates 4D formulations interpreted as three spatial dimensions plus one temporal dimension.
This enables a unified treatment of the time-dependent $H(\textbf{grad})$, $H(\textbf{curl})$, and $H(\text{div})$ problems within a single spatiotemporal equation. A further distinctive feature of this 4D formulation is that it also incorporates the purely temporal problem \eqref{sys: conv_diff_temporal}.

\subsection{Spatiotemporal exterior calculus}\label{subsec: form&hodge}

We introduce differential forms and Hodge star operators suitable for the convection-diffusion problems \eqref{sys: conv_diff_grad}--\eqref{sys: conv_diff_temporal} in a 4D space-time setting (i.e., 3D space + 1D time), employing the exterior calculus framework (e.g., \cite{tu2011manifolds}).
Let $\Lambda^k$ denote the space of differential $k$-forms on the space-time domain $\Omega$. 
Each element of $\Lambda^k$ can be expressed locally as a linear combination of the basis forms $\bm{\omega}_k$ using the standard wedge product, denoted by $\wedge$.  These combinations are listed in \Cref{tab: Hodge_star}.

\vskip 3pt
\textit{\textbf{Hodge star operators.}} We next recall the Hodge star operator, 
which establishes a correspondence between differential $k$-forms and $(4-k)$-forms in space-time and provides the metric structure necessary for defining inner products and adjoint operators. 
In the 4D setting, 
we work under a Euclidean metric with signature $(+,+,+,+)$ 
using coordinates $(x,y,z,t)$. The standard Hodge star operator \(*\) and its scaled variant \(*_\alpha\) are summarized in \Cref{tab: Hodge_star}.
\begin{table}[htbp]
\footnotesize
\caption{Basis $k$-forms and corresponding standard and scaled Hodge star operators in 4D.}\label{tab: Hodge_star}
\begin{center}
  \begin{tabular}{|c||r|r|r|} \hline
   $k$ & $\bm{\omega}_k$ & $*\,\bm{\omega}_k$ & $*_\alpha\bm{\omega}_k$  \\ 
   \hline
   \hline
   0 & $1$ & $\dx\wedge\dy\wedge\dz\wedge\dt$ & $\alpha\,\dx\wedge\dy\wedge\dz\wedge\dt$ \\
   \hline
    $1$ & $\dx$ & $\dy\wedge\dz\wedge\dt$ & $\alpha\,\dy\wedge\dz\wedge\dt$ \\
     & $\dy$ & $\dz\wedge\dx\wedge\dt$ & $\alpha\,\dz\wedge\dx\wedge\dt$ \\
     & $\dz$ & $\dx\wedge\dy\wedge\dt$ & $\alpha\,\dx\wedge\dy\wedge\dt$ \\
      & $\dt$ & $-\dx\wedge\dy\wedge\dz$ & $-\varepsilon\,\dx\wedge\dy\wedge\dz$ \\
      \hline

       2 & $\dy\wedge\dz$ & $\dx\wedge\dt$ & $\alpha\,\dx\wedge\dt$ \\
       & $\dz\wedge\dx$ & $\dy\wedge\dt$ & $\alpha\,\dy\wedge\dt$ \\
         & $\dx\wedge\dy$ & $\dz\wedge\dt$ & $\alpha\,\dz\wedge\dt$ \\
       & $\dx\wedge\dt$ & $\dy\wedge\dz$ & $\varepsilon\,\dy\wedge\dz$ \\
       & $\dy\wedge\dt$ & $\dz\wedge\dx$ & $\varepsilon\,\dz\wedge\dx$ \\
       & $\dz\wedge\dt$ & $\dx\wedge\dy$ & $\varepsilon\,\dx\wedge\dy$ \\
      \hline
      3 & $\dx\wedge\dy\wedge\dz$ & $\dt$ & $\alpha\,\dt$ \\
       & $\dy\wedge\dz\wedge\dt$ & $-\dx$ & $-\varepsilon\,\dx$ \\
       & $\dz\wedge\dx\wedge\dt$ & $-\dy$ & $-\varepsilon\,\dy$ \\
       & $\dx\wedge\dy\wedge\dt$ & $-\dz$ & $-\varepsilon\,\dz$ \\
      \hline
      4 & $\dx\wedge\dy\wedge\dz\wedge\dt$ & $1$ & $\varepsilon$ \\
    \hline
  \end{tabular}
\end{center}
\end{table}
\begin{remark}\label{remark: Hodge_star}
    In the scaled Hodge star operator \(*_\alpha\), a consistent scaling rule is applied as follows: if the basis form $\bm{\omega}_k$ does not involve the temporal component \(\dt\), the wedge product in $*_\alpha\bm{\omega}_k$ is multiplied by the spatial diffusion coefficient \(\alpha\); if it includes \(\dt\), it is instead scaled by the temporal coefficient \(\varepsilon\).
This directional scaling reflects the anisotropy of space-time diffusion, in which the spatial and temporal dimensions contribute differently.
Additionally, the action of the scaled double Hodge star reduces to a scalar
multiplication:
 \[(-1)^{k(4-k)}**_\alpha\,\bm{\omega}_k=\begin{cases}
        \alpha \bm{\omega}_k,&\text{if $\bm{\omega}_k$ does not involve $\mathrm{d}t$},\\[0.5ex]
        \varepsilon \bm{\omega}_k,&\text{if $\bm{\omega}_k$ includes $\mathrm{d}t$}.
    \end{cases}\]
\end{remark}

\textit{\textbf{Spaces of differential forms.}} We denote by $L^2\Lambda^k(\Omega)$ the space of square-integrable differential $k$-forms,
whose coefficients in local coordinates belong to $L^2(\Omega)$.
The corresponding Sobolev space of square-integrable $k$-forms 
whose exterior derivatives are also square-integrable is defined as
\[
H\Lambda^k(\Omega)
:= \{\mathbf{v}_k\in L^2\Lambda^k(\Omega)\;:\;\mathrm{d}_k\mathbf{v}_k\in L^2\Lambda^{k+1}(\Omega)\},
\]
where $\mathrm{d}_k:\Lambda^k\to\Lambda^{k+1}$ denotes the exterior derivative.
Similarly, we introduce the adjoint space
\[
H^*\Lambda^k(\Omega)
:= \{\mathbf{v}_k\in L^2\Lambda^k(\Omega)\;:\;\delta^{k-1}\mathbf{v}_k\in L^2\Lambda^{k-1}(\Omega)\},
\]
where $\delta^{k-1}: \Lambda^k\to\Lambda^{k-1}$ denotes the codifferential operator, defined as the $L^2$-adjoint of $\mathrm{d}_{k-1}$ with respect to the inner product 
induced by the Hodge star operator, $\delta^{k-1}= (-1)*\left(\mathrm{d}_{4-k}\right)*$.
The weighted codifferentials are then defined using the scaled Hodge star,
\begin{equation}
    \delta^{k-1}_{1\alpha} := (-1) * \left( \mathrm{d}_{4 - k} \right) *_\alpha \quad \text{and} \quad
\delta^{k-1}_{\alpha1} := (-1) *_\alpha \left( \mathrm{d}_{4-k} \right) *.\label{eqn: defi_delta}
\end{equation}

\vskip 3pt
\textit{\textbf{Representation of the solution and source forms.}}
We next define how the physical unknowns and source terms from problems \eqref{sys: conv_diff_grad}--\eqref{sys: conv_diff_temporal} 
are represented as differential forms in the 4D setting.
A differential $k$-form solution $\mathbf{u}_k$ and source $\mathbf{f}_k$ on the domain $\Omega$ with coordinates $(x,y,z,t)$ are defined as follows:
\vskip 2pt
\begin{itemize}
    \item $k=0$ (0-form):  Here, $\mathbf{u}_0 := u(x,y,z,t)$ and $\mathbf{f}_0:=f(x,y,z,t)$, represent scalar fields.
    \vskip 2pt
    \item $k=1$ (1-form): Let $\bu = \langle u_1(x,y,z,t), u_2(x,y,z,t), u_3(x,y,z,t)\rangle$ and $\mathbf{f} = \langle f_1(x,y,z,t), f_2(x,y,z,t), f_3(x,y,z,t)\rangle$ denote the spatial solution and source vector fields, respectively. The corresponding 1-form representations are
    \begin{align*}
        \mathbf{u}_1 &:= u_1\,\dx + u_2\,\dy + u_3\,\dz + 0\cdot\dt,\\
        \mathbf{f}_1 &:= f_1\,\dx + f_2\,\dy + f_3\,\dz + 0\cdot\dt.
    \end{align*}
    \vskip 2pt
    \item $k=2$ (2-form): With the same spatial vector fields $\mathbf{u}$ and $\mathbf{f}$, the corresponding 2-form representations are
    \begin{align*}
        \mathbf{u}_2 &:= u_1\,\dy\wedge\dz+u_2\,\dz\wedge\dx+u_3\,\dx\wedge\dy\\
        &\qquad\qquad+0\cdot\dx\wedge\dt+0\cdot\dy\wedge\dt+0\cdot\dz\wedge\dt,\\
        \mathbf{f}_2 &:= f_1\,\dy\wedge\dz+f_2\,\dz\wedge\dx+f_3\,\dx\wedge\dy\\
        &\qquad\qquad+0\cdot\dx\wedge\dt+0\cdot\dy\wedge\dt+0\cdot\dz\wedge\dt.
    \end{align*}
    \item $k=3$ (3-form): For scalar functions $u$ and $f$, the 3-form representations are
    \begin{align*}
        \mathbf{u}_3 &:=u\,\dx\wedge\dy\wedge\dz + 0\cdot\dy\wedge\dz\wedge\dt + 0\cdot\dz\wedge\dx\wedge\dt + 0\cdot\dx\wedge\dy\wedge\dt,\\
        \mathbf{f}_3 &:=f\,\dx\wedge\dy\wedge\dz + 0\cdot\dy\wedge\dz\wedge\dt + 0\cdot\dz\wedge\dx\wedge\dt + 0\cdot\dx\wedge\dy\wedge\dt.
    \end{align*}
    \item $k=4$ (4-form): Here, $\mathbf{u}_4:=0\cdot \dx\wedge\dy\wedge\dz\wedge\dt$ and $\mathbf{f}_4:=0\cdot \dx\wedge\dy\wedge\dz\wedge\dt$, are the vanishing top-degree forms in this construction.
\end{itemize}
\begin{remark}\label{remark: rule_for_uk}
    A consistent rule governs the construction of these differential forms: for any basis wedge product that does not involve the time differential $\dt$, the corresponding coefficient is taken from the scalar fields $u$ and $f$, or from the components 
of the spatial vector fields $\mathbf{u} = \langle u_1, u_2, u_3 \rangle$ and $\mathbf{f} = \langle f_1, f_2, f_3 \rangle$. In contrast, if the wedge product includes $\dt$, the associated coefficient is set to zero.
\end{remark}

\vskip 3pt
\textit{\textbf{Weighted Hodge Laplacian.}}
Finally, the weighted Hodge Laplacian is a second-order differential operator acting on differential forms. It extends the classical Laplacian to the framework of exterior calculus and is defined using the exterior derivative $\mathrm{d}_k$ together
 with the weighted codifferentials $\delta^k_{1\alpha}$ and $\delta^k_{\alpha1}$:
\begin{equation}
    \Delta_{\alpha,k} = \delta^k_{1\alpha}\mathrm{d}_k + \mathrm{d}_{k-1}\delta^{k-1}_{\alpha1}.\label{eqn: Hodge_laplacian}
\end{equation}
This Hodge Laplacian primarily represents diffusion. Our formulation of convection-diffusion problems begins by introducing appropriate modifications to this operator.

\section{A unified spatiotemporal formulation}\label{sec: unified_spatiotemporal}

 Next, using the differential-forms framework introduced above, we present the unified spatiotemporal formulation for the time-dependent convection-diffusion problems \eqref{sys: conv_diff_grad}--\eqref{sys: conv_diff_temporal}.
 To incorporate convection effects, we introduce the spatiotemporal convection 1-form by extending the spatial convection field $\bm{\beta} = \langle \beta_1, \beta_2, \beta_3 \rangle$ into the 4D setting, together with the spatial diffusion coefficient $\alpha$ and the artificial perturbation parameter $\varepsilon$:
\begin{equation}
\mathbf{b}_1 = \alpha^{-1}\beta_1\,\mathrm{d}x + \alpha^{-1}\beta_2\,\mathrm{d}y + \alpha^{-1}\beta_3\,\mathrm{d}z - \varepsilon^{-1}\,\mathrm{d}t.
    \label{eqn: spatiotemporal_convection}
\end{equation}
Then, we define the spatiotemporal convection-diffusion flux operator by
\begin{equation}
\mathrm{J}_k\mathbf{u}_k= \mathrm{d}_k\mathbf{u}_k + \mathbf{b}_1 \wedge \mathbf{u}_k ,\label{eqn: defi_flux_J}
\end{equation}
which maps a $k$-form to a $(k+1)$-form.
This operator naturally combines diffusion, represented by $\mathrm{d}_k$, and convection, represented by the wedge product with $\mathbf{b}_1$, into a single differential formulation.
The unified spatiotemporal formulation is obtained by modifying the Hodge Laplacian \eqref{eqn: Hodge_laplacian} with the convection-diffusion flux operator \eqref{eqn: defi_flux_J}. This reformulation yields a unified spatiotemporal equation with a distinctive structure that preserves \textit{key physical properties}.

For each $k=0,1,2,3,4$, the unified spatiotemporal equation is
\begin{equation}
    \left(\delta^{k}_{1\alpha}{\rm{J}}_k + {\rm{d}}_{k-1}\delta^{k-1}_{\alpha1}\right)\mathbf{u}_k=\mathbf{f}_k.\label{eqn: physic_preserving_equation}
\end{equation}
This formulation extends the Hodge-Laplacian structure \eqref{eqn: Hodge_laplacian}
by incorporating the convection contribution $\delta^k_{1\alpha}(\mathbf{b}_1\wedge\mathbf{u}_k)$.
Expanding \eqref{eqn: physic_preserving_equation}, we obtain
\begin{equation*}
    \delta^{k}_{1\alpha}\mathrm{d}_k\mathbf{u}_k + \delta^k_{1\alpha}(\mathbf{b}_1\wedge\mathbf{u}_k) + {\rm{d}}_{k-1}\delta^{k-1}_{\alpha1}\mathbf{u}_k=\mathbf{f}_k.
\end{equation*}
We examine each term explicitly for fixed $k$ in the following subsections.

\subsection*{0-form ($H(\textbf{grad})$) problem}

Let $\mathbf{u}_0 = u$ represent the scalar solution and let $\mathbf{f}_0 = f$ denote the source term. Then, the unified equation \eqref{eqn: physic_preserving_equation} becomes
    \begin{equation*}
        \underbrace{-\varepsilon u_{tt} - \nabla\cdot(\alpha\nabla u)}_{\delta^0_{1\alpha}\mathrm{d}_0\mathbf{u}_0} \ +\  \underbrace{u_t - \nabla\cdot(\bm{\beta}u)}_{\delta^0_{1\alpha}(\mathbf{b}_1\wedge\mathbf{u}_0)} \ + \underbrace{0}_{\mathrm{d}_{-1}\delta^{-1}_{\alpha1}\mathbf{u}_0}= \underbrace{f}_{\mathbf{f}_0},
    \end{equation*}
    which simplifies to the equation:
    \begin{equation}
        -\varepsilon u_{tt} + u_t - \nabla\cdot(\alpha\nabla u+\bm{\beta}u) = f,\label{eqn: physics_Hgrad_equation}
    \end{equation}
    where, again, $\varepsilon>0$ is a small artificial parameter, introduced to eliminate the degeneracy in the spatiotemporal diffusion tensor (see also \cite{bank2017arbitrary} and \Cref{subsec: Spatiotemporal_Hgrad}). This result shows that, with the 0-form solution and source, the unified equation \eqref{eqn: physic_preserving_equation} recovers the spatiotemporal $H(\textbf{grad})$ convection-diffusion equation \eqref{eqn: spatiotemporal-grad-equation}.

\subsection*{1-form ($H(\textbf{curl})$) problem}
Here, $\bu = \langle u_1, u_2, u_3 \rangle$ is the spatial vector solution, and $\mathbf{f} = \langle f_1,f_2,f_3\rangle$ is the source term. The associated 1-form solution is defined as
    \begin{equation*}
        \mathbf{u}_1 = u_1\,\dx + u_2\,\dy + u_3\,\dz + 0\cdot\dt,
    \end{equation*}
    and the corresponding 1-form source term is
    \begin{equation*}
        \mathbf{f}_1 = f_1\,\dx + f_2\,\dy + f_3\,\dz + 0\cdot\dt.
    \end{equation*}
\Cref{tab: physics_Hcurl_case} summarizes the contributions of each basis form in the unified equation \eqref{eqn: physic_preserving_equation}.

\vskip -6pt

\begin{table}[htbp]
\footnotesize
\caption{Component-wise representation of the unified equation for 1-forms.}\label{tab: physics_Hcurl_case}
\begin{center}
  \begin{tabular}{|c||c|c|c|c|c|} \hline
   Basis & $\mathbf{u}_1$ &$\delta^1_{1\alpha}\done\mathbf{u}_1$ & $\delta^1_{1\alpha}(\mathbf{b}_1\wedge\mathbf{u}_1)$ & $\dzero\delta^0_{\alpha1}\mathbf{u}_1$ & $\mathbf{f}_1$ \\ 
   \hline
   \hline
    $\dx$ & $u_1$ & $-\varepsilon \partial_{tt}u_1 + [\nabla\times(\alpha\nabla\times\bu)]_1$ & $\partial_{t}u_1 +[ \nabla\times(\bm{\beta}\times\bu)]_1$ & $-\partial_x(\varepsilon\nabla\cdot\bu)$ & $f_1$\\
    
    $\dy$ & $u_2$ & $-\varepsilon \partial_{tt}u_2 + [\nabla\times(\alpha\nabla\times\bu)]_2$ & $\partial_{t}u_2 +[ \nabla\times(\bm{\beta}\times\bu)]_2$ & $-\partial_y(\varepsilon\nabla\cdot\bu)$ & $f_2$\\
    
    $\dz$ & $u_3$ & $-\varepsilon \partial_{tt}u_3 + [\nabla\times(\alpha\nabla\times\bu)]_3$ & $\partial_{t}u_3 +[ \nabla\times(\bm{\beta}\times\bu)]_3$ & $-\partial_z(\varepsilon\nabla\cdot\bu)$ & $f_3$\\

    \hline
    
    $\dt$ & $0$ & $\varepsilon (\partial_x\partial_tu_1+\partial_y\partial_tu_2+\partial_z\partial_tu_3)$ & $-(\partial_xu_1+\partial_yu_2+\partial_zu_3)$ & $-\partial_t(\varepsilon\nabla\cdot\bu)$ & $0$\\
    \hline
  \end{tabular}
\end{center}
\end{table}

By collecting the spatial terms ($\dx$, $\dy$, $\dz$), the unified equation \eqref{eqn: physic_preserving_equation} yields the corresponding vector-valued form,
\begin{align*}
        \underbrace{-\varepsilon\bu_{tt} + \nabla\times(\alpha\nabla\times\bu)}_{\delta^1_{1\alpha}\done\mathbf{u}_1} \ +\  \underbrace{\bu_t + \nabla\times(\bm{\beta}\times\bu)}_{\delta^1_{1\alpha}(\mathbf{b}_1\wedge\mathbf{u}_1)} 
        \ +\ \underbrace{\left(-\nabla(\varepsilon\nabla\cdot\bu)\right)}_{\dzero\delta^0_{\alpha1}\mathbf{u}_1}\ = \underbrace{\mathbf{f}}_{\mathbf{f}_1}.
    \end{align*}
The $\dt$-component additionally gives
\begin{equation*}
        \underbrace{\varepsilon(\nabla\cdot\bu_t)}_{\delta_{1\alpha}^1\done\mathbf{u}_1} \ +\  \underbrace{(-\nabla\cdot\bu)}_{\delta_{1\alpha}^1(\mathbf{b}_1\wedge\mathbf{u}_1)} \ +\  \underbrace{(-(\varepsilon\nabla\cdot\bu)_t)}_{\dzero\delta_{\alpha1}^0\mathbf{u}_1} \ = \underbrace{0}_{\mathbf{f}_1}.
    \end{equation*}
Assuming sufficient regularity so that $(\nabla \cdot \bu_t) = (\nabla \cdot \bu)_t$, we obtain the following vector-valued system:
\begin{equation}
\begin{cases}
    -\varepsilon\bu_{tt} + \bu_t+ \nabla\times(\alpha\nabla\times\bu +\bm{\beta}\times \bu) - \nabla(\varepsilon\nabla\cdot\bu)=\mathbf{f}, \\
    \nabla\cdot\bu = 0.\label{eqn: physics_Hcurl_system}
\end{cases}
    \end{equation}
The second equation in the system \eqref{eqn: physics_Hcurl_system} enforces the divergence-free condition on the vector field $\bu$, consistent with Gauss’s law for magnetism.  When substituted into the first equation, it yields an $H(\textbf{curl})$ convection-diffusion equation with an artificial temporal perturbation, as described in the $0$-form case. Consequently, the unified spatiotemporal equation \eqref{eqn: physic_preserving_equation}, together with the 1-form solution and source, embodies a physics-preserving spatiotemporal formulation of the $H(\textbf{curl})$ problem.

\subsection*{2-form ($H(\textnormal{div})$) problem}
Again, let $\bu = \langle u_1, u_2, u_3 \rangle$ and $\mathbf{f} = \langle f_1, f_2, f_3 \rangle$. We define the associated 2-form solution and source term as
    \begin{align*}
        \mathbf{u}_2 &= u_1 \,{\rm{d}}y\wedge{\rm{d}}z+u_2 \,{\rm{d}}z\wedge{\rm{d}}x+u_3 \,{\rm{d}}x\wedge{\rm{d}}y\\
        &\qquad\qquad+0\cdot{\rm{d}}x\wedge{\rm{d}}t+0 \cdot{\rm{d}}y\wedge{\rm{d}}t+0 \cdot{\rm{d}}z\wedge{\rm{d}}t,\\
         \mathbf{f}_2 &= f_1 \,{\rm{d}}y\wedge{\rm{d}}z+f_2 \,{\rm{d}}z\wedge{\rm{d}}x+f_3 \,{\rm{d}}x\wedge{\rm{d}}y\\
        &\qquad\qquad+0\cdot{\rm{d}}x\wedge{\rm{d}}t+0 \cdot{\rm{d}}y\wedge{\rm{d}}t+0 \cdot{\rm{d}}z\wedge{\rm{d}}t,
    \end{align*}
and summarize the contributions of each basis form in the unified equation \eqref{eqn: physic_preserving_equation} in \Cref{tab: physics_Hdiv_case}.

\vskip -6pt

\begin{table}[htbp]
\footnotesize
\caption{Component-wise representation of the unified equation for 2-forms.}\label{tab: physics_Hdiv_case}
\begin{center}
  \begin{tabular}{|c||c|c|c|c|c|} \hline
   Basis & $\mathbf{u}_2$ &$\delta_{1\alpha}^2\dtwo\mathbf{u}_2$ & $\delta_{1\alpha}^2(\mathbf{b}_1\wedge\mathbf{u}_2)$ & $\done\delta_{\alpha1}^1\mathbf{u}_2$ & $\mathbf{f}_2$ \\ 
   \hline
   \hline
    $\dy\wedge\dz$ & $u_1$ & $-\varepsilon \partial_{tt}u_1 - \partial_x(\alpha\nabla\cdot\bu)$ & $\partial_{t}u_1 - \partial_x(\bm{\beta}\cdot\bu)$ & $\partial_y(\varepsilon\nabla\times\bu)_3-\partial_z(\varepsilon\nabla\times\bu)_2$ & $f_1$\\
    
    $\dz\wedge\dx$ & $u_2$ & $-\varepsilon \partial_{tt}u_2 - \partial_y(\alpha\nabla\cdot\bu)$ & $\partial_{t}u_2 - \partial_y(\bm{\beta}\cdot\bu)$ & $\partial_z(\varepsilon\nabla\times\bu)_1-\partial_x(\varepsilon\nabla\times\bu)_3$ & $f_2$\\
    
    $\dx\wedge\dy$ & $u_3$ & $-\varepsilon \partial_{tt}u_3 - \partial_z(\alpha\nabla\cdot\bu)$ & $\partial_{t}u_3 - \partial_z(\bm{\beta}\cdot\bu)$ & $\partial_x(\varepsilon\nabla\times\bu)_2-\partial_y(\varepsilon\nabla\times\bu)_1$ & $f_3$\\

    \hline
    
    $\dx\wedge\dt$ & $0$ & $\varepsilon (\partial_y\partial_t u_3 - \partial_z\partial_t u_2)$ & $-(\partial_yu_3 - \partial_zu_2)$ & $-\partial_t(\varepsilon\nabla\times\bu)_1$ & $0$\\

    $\dy\wedge\dt$ & $0$ &$\varepsilon (\partial_z\partial_t u_1 - \partial_x\partial_t u_3)$ & $-(\partial_zu_1-\partial_xu_3)$ & $-\partial_t(\varepsilon\nabla\times\bu)_2$ & $0$\\

    $\dz\wedge\dt$ &  $0$ &$\varepsilon (\partial_x\partial_t u_2 - \partial_y\partial_t u_1)$ & $-(\partial_xu_2 - \partial_yu_1)$ & $-\partial_t(\varepsilon\nabla\times\bu)_3$ & $0$\\
    \hline
  \end{tabular}
\end{center}
\end{table}

From the spatial components $\dy \wedge \dz$, $\dz \wedge \dx$, and $\dx \wedge \dy$, the unified equation \eqref{eqn: physic_preserving_equation} becomes
\begin{equation*}
        \underbrace{-\varepsilon\bu_{tt}-\nabla(\alpha\nabla\cdot\bu)}_{\delta_{1\alpha}^2\dtwo\mathbf{u}_2} \ +\  \underbrace{\bu_t - \nabla(\bm{\beta}\cdot\bu)}_{\delta_{1\alpha}^2(\mathbf{b}_1\wedge\mathbf{u}_2)} \ +\  \underbrace{\nabla\times(\varepsilon\nabla\times\bu)}_{\done\delta_{\alpha1}^1\mathbf{u}_2} \ = \underbrace{\mathbf{f}}_{\mathbf{f}_2}.
    \end{equation*}
From the temporal-spatial components $\dx \wedge \dt$, $\dy \wedge \dt$, and $\dz \wedge \dt$, we have
\begin{equation*}
        \underbrace{\varepsilon\nabla\times \bu_t}_{\delta_{1\alpha}^2\dtwo\mathbf{u}_2} \ +\  \underbrace{(-\nabla\times\bu)}_{\delta_{1\alpha}^2(\mathbf{b}_1\wedge\mathbf{u}_2)} \ +\  \underbrace{(-(\varepsilon\nabla\times\bu)_t)}_{\done\delta_{\alpha1}^1\mathbf{u}_2}\ =\underbrace{\mathbf{0}}_{\mathbf{f}_2}.
    \end{equation*}
Again, assuming sufficient regularity to commute time and space derivatives, the resulting vector-valued system reads
\begin{equation}
\begin{cases}
    -\varepsilon\bu_{tt}+\bu_t-\nabla(\alpha\nabla\cdot\bu+\bm{\beta}\cdot\bu) + \nabla\times(\varepsilon\nabla\times\bu)=\mathbf{f}, \\
    \nabla\times\bu = \mathbf{0}.\label{eqn: physics_Hdiv_system}
\end{cases}
    \end{equation}
    Similar to the $1$-form case, the second equation in the system \eqref{eqn: physics_Hdiv_system} enforces a curl-free condition on the vector field $\bu$, which naturally arises in $H(\textnormal{div})$ elliptic problems. Substituting this constraint into the first equation yields an $H(\textnormal{div})$ convection-diffusion equation with an artificial temporal perturbation.

\subsection*{3-form ($L^2$-temporal) problem}
Let $u = u(x, y, z, t)$ be a scalar solution and $f = f(x, y, z, t)$ be the source term. The associated 3-form and source term are 
    \begin{align*}
        \mathbf{u}_3 &= u\,\dx\wedge\dy\wedge\dz + 0\cdot\dy\wedge\dz\wedge\dt + 0\cdot\dz\wedge\dx\wedge\dt + 0\cdot\dx\wedge\dy\wedge\dt,\\
        \mathbf{f}_3 &= f\,\dx\wedge\dy\wedge\dz + 0\cdot\dy\wedge\dz\wedge\dt + 0\cdot\dz\wedge\dx\wedge\dt + 0\cdot\dx\wedge\dy\wedge\dt.
    \end{align*}  
\Cref{tab: physics_L2_case} summarizes the contributions of each basis form in the unified equation \eqref{eqn: physic_preserving_equation}.

\vskip -6pt

\begin{table}[htbp]
\footnotesize
\caption{Component-wise representation of the unified equation for 3-forms.}\label{tab: physics_L2_case}
\begin{center}
  \begin{tabular}{|c||c|c|c|c|c|} \hline
   Basis & $\mathbf{u}_3$ &$\delta_{1\alpha}^3\dthree\mathbf{u}_3$ & $\delta_{1\alpha}^3(\mathbf{b}_1\wedge\mathbf{u}_3)$ & $\dtwo\delta_{\alpha1}^2\mathbf{u}_3$ & $\mathbf{f}_3$ \\ 
   \hline
   \hline
    $\dx\wedge\dy\wedge\dz$ & $u$ & $-\varepsilon \partial_{tt}u$ & $\partial_{t}u$ & $-\varepsilon(\partial_{xx}u+\partial_{yy}u+\partial_{zz}u)$ & $f$\\

    \hline
    $\dy\wedge\dz\wedge\dt$ & $0$ & $\varepsilon\partial_x\partial_t u$ & $-\partial_xu$ & $-\partial_t(\varepsilon\partial_xu)$ & $0$\\
    
    $\dz\wedge\dx\wedge\dt$ & $0$ & $\varepsilon\partial_y\partial_t u$ & $-\partial_yu$ & $-\partial_t(\varepsilon\partial_yu)$ & $0$\\
    
    $\dx\wedge\dy\wedge\dt$ & $0$ & $\varepsilon\partial_z\partial_t u$ & $-\partial_zu$ & $-\partial_t(\varepsilon\partial_zu)$ & $0$\\
    \hline
  \end{tabular}
\end{center}
\end{table}

From the spatial component $\dx \wedge \dy \wedge \dz$, the unified equation \eqref{eqn: physic_preserving_equation} becomes
    \begin{equation*}
        \underbrace{-\varepsilon u_{tt}}_{\delta_{1\alpha}^3\dthree\mathbf{u}_3} \ + \underbrace{u_t}_{\delta_{1\alpha}^3(\mathbf{b}_1\wedge\mathbf{u}_3)}  +\ \underbrace{(-\nabla\cdot(\varepsilon \nabla u))}_{\dtwo\delta_{\alpha1}^2\mathbf{u}_3} \ = \underbrace{f}_{\mathbf{f}_3}.
    \end{equation*}
The remaining temporal-spatial components ($\dy \wedge \dz \wedge \dt$, etc.) yield
\begin{equation*}
        \underbrace{\nabla(\varepsilon u_t)}_{\delta_{1\alpha}^3\dthree\mathbf{u}_3}\ +\ \underbrace{(-\nabla u)}_{\delta_{1\alpha}^3(\mathbf{b}_1\wedge\mathbf{u}_3)}\ +\ \underbrace{(-(\varepsilon\nabla u)_t)}_{\dtwo\delta_{\alpha1}^2\mathbf{u}_3} \ = \underbrace{\mathbf{0}}_{\mathbf{f}_3}.
    \end{equation*}
    Once again, assuming sufficient temporal-spatial regularity, the resulting scalar-valued system reads
\begin{equation}
\begin{cases}
    -\varepsilon u_{tt} + u_t - \nabla\cdot(\varepsilon\nabla u) = f, \\
    \nabla u = \mathbf{0}.\label{eqn: physics_L2_system}
\end{cases}
    \end{equation}
The second equation in the system \eqref{eqn: physics_L2_system} enforces a gradient-free condition on the scalar solution $u$, meaning that $u$ must be spatially constant. Substituting this constraint into the first equation yields an artificially perturbed first-order evolution equation. Interestingly, the unified spatiotemporal equation \eqref{eqn: physic_preserving_equation} 
also admits a spatiotemporal formulation of the $L^2$-temporal problem (\eqref{sys: conv_diff_temporal} in the case $\varepsilon = 0$), which is a distinctive feature of the 3-form setting in four dimensions.

\subsection*{4-form case}
Following the consistent construction rule described in \Cref{remark: rule_for_uk}, the associated 4-form solution and source term are defined as
\begin{align*}
    \mathbf{u}_4 &= 0\cdot\dx\wedge\dy\wedge\dz\wedge\dt,\\
    \mathbf{f}_4 &= 0\cdot\dx\wedge\dy\wedge\dz\wedge\dt.
\end{align*}
The unified equation \eqref{eqn: physic_preserving_equation} in this case becomes
    \begin{equation*}
        \underbrace{0}_{\delta_{1\alpha}^4\mathrm{d}_4\mathbf{u}_4} \ +  \underbrace{0}_{\delta_{1\alpha}^4(\mathbf{b}_1\wedge\mathbf{u}_4)}  + \ \underbrace{(-\partial_x(\alpha\cdot0)-\partial_y(\alpha\cdot0)-\partial_z(\alpha\cdot0)-\partial_t(\varepsilon\cdot0))}_{\mathrm{d}_{3}\delta_{\alpha1}^3\mathbf{u}_4}\ = \underbrace{0}_{\mathbf{f}_4}.
    \end{equation*}
The first two terms vanish trivially since both the exterior derivative of $\mathbf{u}_4$ and its wedge product with $\mathbf{b}_1$ are zero. The third term, $\mathrm{d}_{3}\delta_{\alpha1}^3\mathbf{u}_4$ (the exact component of the Hodge Laplacian), also vanishes due to the zero-valued definition of $\mathbf{u}_4$. Finally, the right-hand side is zero by construction, as $\mathbf{f}_4$ is identically zero. Hence, the unified equation \eqref{eqn: physic_preserving_equation} holds in the 4-form case.

\begin{remark}
    In summary, the spatiotemporal formulation of all time-dependent convection-diffusion problems \eqref{sys: conv_diff_grad}--\eqref{sys: conv_diff_temporal} can be expressed through the single unified equation \eqref{eqn: physic_preserving_equation} within the framework of spatiotemporal exterior calculus. This unified formulation inherently preserves the underlying physics according to the differential-form setting. Moreover, while retaining the advantages of stationary problems, this structure enables unified analysis, supports the development of unified numerical methods, and ensures that physical constraints are preserved within the equation itself, without requiring additional conditions.
\end{remark}

\subsection{Boundary conditions}\label{subsec: boundary_conditions}

We now introduce spatiotemporal boundary conditions in exterior calculus notation. These conditions are constructed to be consistent with the original boundary and initial conditions in \eqref{sys: conv_diff_grad}--\eqref{sys: conv_diff_temporal}. 
Let $\mathbf{n}=\langle n_1,n_2,n_3\rangle$ denote the unit outward normal vector on the spatial boundary $\Gamma_\mathbf{x} = \partial\Omega_\mathbf{x}\times\Omega_t$. We define the associated spatiotemporal normal 1-forms:
\begin{align*}
    &\mathbf{n}_{\mathbf{x}} := n_1\,\dx + n_2\,\dy + n_3\,\dz + 0\cdot\dt\quad\text{on }\Gamma_\mathbf{x},\\
    &\mathbf{n}_{t_0} := 0\cdot\dx + 0\cdot\dy + 0\cdot\dz + (-1)\cdot\dt\quad\text{on }\Gamma_{t_0},\\
    &\mathbf{n}_{T} := 0\cdot\dx + 0\cdot\dy + 0\cdot\dz + 1\cdot\dt\quad\text{on }\Gamma_{T},
\end{align*}
where $\Gamma_{t_0} = \Omega_\mathbf{x}\times\{t=t_{0}\}$ and $\Gamma_{T} = \Omega_\mathbf{x}\times\{t=T\}$.
We also define the interior product with respect to $\mathbf{n}_{T}$ by
\begin{align*}
\begin{cases}
    \iota_{\mathbf{n}_{T}}(w\,\mathrm{d}x^{i_1}\wedge\mathrm{d}x^{i_2}\wedge\cdots\wedge\mathrm{d}x^{i_k}\wedge\mathrm{d}t) = w\,\mathrm{d}x^{i_1}\wedge\mathrm{d}x^{i_2}\wedge\cdots\wedge\mathrm{d}x^{i_k},\\[0.5ex]
    \iota_{\mathbf{n}_{T}}(\mathrm{d}x^{i_1}\wedge\mathrm{d}x^{i_2}\wedge\cdots\wedge\mathrm{d}x^{i_k}) = 0,
\end{cases}
\end{align*}
where $w$ is a function, and $x^{i_j}$ are spatial coordinates (see \cite{tu2011manifolds} for more details).
Using these operators, the spatiotemporal boundary conditions for the unified physics-preserving equation \eqref{eqn: physic_preserving_equation} are defined on $\partial\Omega = \Gamma_{\mathbf{x}}\cup\Gamma_{t_0}\cup\Gamma_{T}$ as follows:
\begin{align}
    \text{Spatial boundary condition ($\mathbf{x}$-BC):}&\quad\mathbf{n}_{\mathbf{x}}\wedge \mathbf{u}_k = 0\quad \text{on } \Gamma_{\mathbf{x}},\label{eqn: xbc}\\
    \text{Temporal boundary condition ($t_0$-BC):}&\quad\mathbf{n}_{t_0}\wedge \mathbf{u}_k = \mathbf{n}_{t_0}\wedge \mathbf{u}_{k,t_0}\quad \text{on } \Gamma_{t_0},\label{eqn: tbc}\\
    \text{Artificial boundary condition ($\varepsilon$-BC):}&\quad \iota_{\mathbf{n}_{T}}(**_\alpha\mathrm{d}_k \mathbf{u}_k) = 0\quad \text{on } \Gamma_{T},\label{eqn: ebc}
\end{align}
where $\mathbf{u}_{k,t_0}$ denotes the initial $k$-form associated with the initial conditions in \eqref{sys: conv_diff_grad}--\eqref{sys: conv_diff_temporal}.  In the following, we examine these boundary conditions for fixed $k$. 

\subsubsection*{0-form problem}
Let $\mathbf{u}_0 = u(\mathbf{x},t)$ and $\mathbf{u}_{0,t_0} = u_{t_0}(\mathbf{x})$ as in \eqref{eqn: grad_initial}. The spatial boundary condition \eqref{eqn: xbc} is computed as
\begin{equation*}
    \mathbf{n}_\mathbf{x}\wedge\mathbf{u}_0 = n_1u\,\dx + n_2u\,\dy + n_3u\,\dz + 0\cdot\dt = 0,
\end{equation*}
which implies $u=0$ on $\Gamma_\mathbf{x}$.  
The temporal boundary condition \eqref{eqn: tbc} reduces to
\begin{equation*}
    \mathbf{n}_{t_0}\wedge\mathbf{u}_0 = - u\,\dt = -u_{t_0}\,\dt = \mathbf{n}_{t_0}\wedge\mathbf{u}_{0,t_0},
\end{equation*}
which yields $u(\mathbf{x},t_0) = u_{t_0}(\mathbf{x})$ on $\Gamma_{t_0}$.  
From \Cref{tab: Hodge_star} and \Cref{remark: Hodge_star},
\begin{equation*}
    **_\alpha\mathrm{d}_0 \mathbf{u}_0 = (-\alpha\partial_xu)\,\dx +(- \alpha\partial_yu)\,\dy +(- \alpha\partial_zu)\,\dz +(- \varepsilon\partial_tu)\,\dt,
\end{equation*}
and applying $\iota_{\mathbf{n}_{T}}$ gives
\begin{equation*}
    \iota_{\mathbf{n}_{T}}(**_\alpha\mathrm{d}_0 \mathbf{u}_0) = -\varepsilon \partial_t u=0.
\end{equation*}
In summary, the boundary conditions \eqref{eqn: xbc}--\eqref{eqn: ebc} for the 0-form case are
\begin{equation*}
\left\{\begin{array}{rl}
   \text{$\mathbf{x}$-BC:}& u=0\quad \text{on } \Gamma_{\mathbf{x}}, \\[0.5ex]
     \text{$t_0$-BC:}& u(\mathbf{x},t_0) = u_{t_0}(\mathbf{x})\quad \text{on } \Gamma_{t_0},\\[0.5ex]
    \text{$\varepsilon$-BC:}& \varepsilon u_t(\mathbf{x},T) = 0 \quad \text{on } \Gamma_{T}.
\end{array}\right.
\end{equation*}
Thus, the spatial ($\mathbf{x}$-BC) and temporal ($t_0$-BC) boundary conditions are consistent with those of the $H(\textbf{grad})$ convection-diffusion problem \eqref{eqn: grad_boundary}--\eqref{eqn: grad_initial}, while the artificial boundary condition ($\varepsilon$-BC) 
corresponds to a terminal velocity induced by the artificial temporal perturbation \eqref{eqn: grad_neumann}, and which is trivially satisfied when $\varepsilon=0$.

\subsubsection*{1-form problem}
Let $\mathbf{u}_1 = u_1\,\dx + u_2\,\dy + u_3\,\dz$ with initial 1-form $\mathbf{u}_{1,t_0}=\mathbf{u}_{t_0}(\mathbf{x})\cdot\mathrm{d}\mathbf{x}$ prescribed in \eqref{eqn: curl_initial}.
The spatial boundary condition \eqref{eqn: xbc} is
\begin{align*}
    \mathbf{n}_\mathbf{x}\wedge\mathbf{u}_1 &= (n_2 u_3 - n_3u_2)\,\dy\wedge\dz +(n_3 u_1 - n_1u_3)\,\dz\wedge\dx\\ 
    &\quad\quad+(n_1 u_2 - n_2u_1)\,\dx\wedge\dy+0\cdot\dx\wedge\dt + 0\cdot\dy\wedge\dt + 0\cdot\dz\wedge\dt = 0,
\end{align*}
which corresponds to the tangential condition $\mathbf{n}\times \mathbf{u} = \mathbf{0}$ on $\Gamma_\mathbf{x}$.
The temporal boundary condition \eqref{eqn: tbc} is
\begin{equation*}
    \mathbf{n}_{t_0}\wedge\mathbf{u}_1 =  u_1\,\dx\wedge\dt + u_2\,\dy\wedge\dt + u_3\,\dz\wedge\dt = \mathbf{n}_{t_0}\wedge\mathbf{u}_{1,t_0},
\end{equation*}
so that $\mathbf{u}(\mathbf{x},t_0) = \mathbf{u}_{t_0}(\mathbf{x})$ on $\Gamma_{t_0}$.
For the artificial boundary condition \eqref{eqn: ebc}, 
\begin{align*}
    **_\alpha\mathrm{d}_1 \mathbf{u}_1 &= \alpha(\partial_yu_3 - \partial_zu_2)\,\dy\wedge\dz + \alpha(\partial_zu_1 - \partial_xu_3)\,\dz\wedge\dx\\
    &\quad\quad+\alpha(\partial_xu_2 - \partial_yu_1)\,\dx\wedge\dy + (-\varepsilon\partial_t u_1)\,\dx\wedge\dt\\
    &\quad\quad+(-\varepsilon\partial_t u_2)\,\dy\wedge\dt + (-
    \varepsilon\partial_t u_3)\,\dz\wedge\dt,
\end{align*}
and applying $\iota_{\mathbf{n}_{T}}$ yields
\begin{equation*}
    \iota_{\mathbf{n}_{T}}(**_\alpha\mathrm{d}_1 \mathbf{u}_1) = (-\varepsilon \partial_t u_1)\,\dx + (-\varepsilon \partial_t u_2)\,\dy + (-\varepsilon \partial_t u_3)\,\dz=0.
\end{equation*}
In summary, the boundary conditions \eqref{eqn: xbc}--\eqref{eqn: ebc} in the 1-form case are
\begin{equation*}
\left\{\begin{array}{rl}
   \text{$\mathbf{x}$-BC:}  &  \mathbf{n}\times\bu=\mathbf{0}\quad \text{on } \Gamma_{\mathbf{x}}, \\[0.5ex]
     \text{$t_0$-BC:}& \mathbf{u}(\mathbf{x},t_0) = \mathbf{u}_{t_0}(\mathbf{x})\quad \text{on } \Gamma_{t_0},\\[0.5ex]
     \text{$\varepsilon$-BC:}& \varepsilon \bu_t(\mathbf{x},T) = \mathbf{0} \quad \text{on } \Gamma_{T}.
\end{array}\right.
\end{equation*}
As in the $H(\textbf{grad})$ case, the spatial ($\mathbf{x}$-BC) and temporal ($t_0$-BC) conditions correspond to the given data in the $H(\textbf{curl})$ problem \eqref{eqn: curl_boundary}--\eqref{eqn: curl_initial}, while the artificial condition ($\varepsilon$-BC) gives a terminal velocity associated with the temporal perturbation.

\subsubsection*{2-form problem}
Let $\mathbf{u}_2 = u_1\,\dy\wedge\dz + u_2\,\dz\wedge\dx + u_3\,\dx\wedge\dy$ with initial 2-form 
$\mathbf{u}_{2,t_0}$ defined analogously.  The spatial boundary condition \eqref{eqn: xbc} becomes
\begin{align*}
    \mathbf{n}_\mathbf{x}\wedge\mathbf{u}_2 = (n_1u_1+n_2u_2+n_3u_3)\,\dx\wedge\dy\wedge\dz=0,
\end{align*}
which corresponds to the normal condition $\mathbf{u}\cdot \mathbf{n} = 0$ on $\Gamma_\mathbf{x}$.  
The temporal boundary condition \eqref{eqn: tbc} is
\begin{equation*}
    \mathbf{n}_{t_0}\wedge\mathbf{u}_2 =  (-u_1)\,\dy\wedge\dz\wedge\dt + (-u_2)\,\dz\wedge\dx\wedge\dt + (-u_3)\,\dx\wedge\dy\wedge\dt,
\end{equation*}
so that $\mathbf{u}(\mathbf{x},t_0) = \mathbf{u}_{t_0}(\mathbf{x})$ on $\Gamma_{t_0}$.  
For the artificial boundary condition \eqref{eqn: ebc},
\begin{align*}
    **_\alpha\mathrm{d}_2 \mathbf{u}_2 &= (-\alpha(\partial_xu_1+\partial_yu_2+\partial_zu_3))\,\dx\wedge\dy\wedge\dz \\
    &\quad\quad+(-\varepsilon\partial_tu_1)\,\dy\wedge\dz\wedge\dt +(-\varepsilon\partial_tu_2)\,\dz\wedge\dx\wedge\dt\\
    &\quad\quad+(-\varepsilon\partial_tu_3)\,\dx\wedge\dy\wedge\dt,
\end{align*}
and applying $\iota_{\mathbf{n}_{T}}$ yields
\begin{equation*}
    \iota_{\mathbf{n}_{T}}(**_\alpha\mathrm{d}_2 \mathbf{u}_2) = (-\varepsilon \partial_t u_1)\,\dy\wedge\dz + (-\varepsilon \partial_t u_2)\,\dz\wedge\dx + (-\varepsilon \partial_t u_3)\,\dx\wedge\dy=0.
\end{equation*}
In summary, the boundary conditions \eqref{eqn: xbc}--\eqref{eqn: ebc} for the 2-form case are
\begin{equation*}
\left\{\begin{array}{rl}
   \text{$\mathbf{x}$-BC:}& \bu\cdot\mathbf{n}=0\quad \text{on } \Gamma_{\mathbf{x}},       \\[0.5ex]
    \text{$t$-BC:}& \mathbf{u}(\mathbf{x},t_0) = \mathbf{u}_{t_0}(\mathbf{x})\quad \text{on } \Gamma_{t_0},\\[0.5ex]
     \text{$\varepsilon$-BC:}& \varepsilon \bu_t(\mathbf{x},T) = \mathbf{0} \quad \text{on } \Gamma_{T}.
\end{array}\right.
\end{equation*}
As in the previous cases, the spatial and temporal conditions match the $H(\textnormal{div})$ data \eqref{eqn: div_boundary}--\eqref{eqn: div_initial}, and the artificial condition specifies the terminal velocity.

\subsubsection*{3-form problem}
Let $\mathbf{u}_3 = u\,\dx\wedge\dy\wedge\dz$ with initial 3-form $\mathbf{u}_{3,t_0} = u_{t_0}\,\dx\wedge\dy\wedge\dz$ as in \eqref{eqn: temporal_initial}.
The spatial boundary condition \eqref{eqn: xbc} becomes
\begin{align*}
    \mathbf{n}_\mathbf{x}\wedge\mathbf{u}_3 = 0\cdot\dx\wedge\dy\wedge\dz\wedge\dt = 0,
\end{align*}
and yields no meaningful spatial boundary condition.
This occurs because the 3-form problem involves only the temporal derivative and thus does not need a spatial boundary condition.
The temporal boundary condition \eqref{eqn: tbc} is
\begin{equation*}
    \mathbf{n}_{t_0}\wedge\mathbf{u}_3 = (-u)\,\dx\wedge\dy\wedge\dz\wedge\dt=\mathbf{n}_{t_0}\wedge\mathbf{u}_{3,t_0},
\end{equation*}
so that $u(\mathbf{x},t_0) = u_{t_0}(\mathbf{x})$ on $\Gamma_{t_0}$.
For the artificial boundary condition \eqref{eqn: ebc}, 
\begin{align*}
    **_\alpha\mathrm{d}_3 \mathbf{u}_3 = (\varepsilon\partial_tu)\,\dx\wedge\dy\wedge\dz\wedge\dt,
\end{align*}
and applying $\iota_{\mathbf{n}_{T}}$ gives
\begin{equation*}
    \iota_{\mathbf{n}_{T}}(**_\alpha\mathrm{d}_3 \mathbf{u}_3) = (\varepsilon\partial_tu)\,\dy\wedge\dy\wedge\dz=0.
\end{equation*}
In summary, the boundary conditions \eqref{eqn: xbc}--\eqref{eqn: ebc} for the 3-form case are
\begin{equation*}
\left\{\begin{array}{rl}
    \text{$\mathbf{x}$-BC:}& \text{not applicable}      ,       \\[0.5ex]
    \text{$t$-BC:}& u(\mathbf{x},t_0) = u_{t_0}(\mathbf{x})\quad \text{on } \Gamma_{t_0},\\[0.5ex]
      \text{$\varepsilon$-BC:}& \varepsilon u_t(\mathbf{x},T) = 0 \quad \text{on } \Gamma_{T}.
\end{array}\right.
\end{equation*}
The temporal condition is consistent with the one in the $L^2$-temporal problem \eqref{eqn: temporal_initial}, and the artificial condition specifies the terminal velocity for the temporal perturbation, which vanishes as $\varepsilon$ goes to $0$.

\section{Variational formulation and well-posedness}\label{sec: well-posed}
In this section, we establish the well-posedness of the unified problem within a variational setting.  To this end, we first introduce an exponentially-fitted spatiotemporal flux operator tailored to the unified convection-diffusion framework, and then derive a variational formulation for the unified equation \eqref{eqn: physic_preserving_equation} together with the boundary conditions \eqref{eqn: xbc}--\eqref{eqn: ebc}.

\subsection{Exponentially-fitted spatiotemporal flux operator}\label{sec: exp_fitted_spatiotemporal}

Exponentially-fitted flux operators are designed to symmetrize fluxes in convection-diffusion problems.  
As a motivating example (see \cite{bank2017arbitrary,lazarov2012exponential} for more details), let $\psi(x,y,z,t)$ be a potential function satisfying $\nabla_\mathbf{y}\psi = D^{-1}\mathbf{b}$. Then, the $H(\textbf{grad})$ convection-diffusion flux $J(u)$ defined in \eqref{eqn: flux} can be expressed as
\begin{equation*}
    J(u) = D\nabla_\mathbf{y} u + \mathbf{b}u = D e^{-\psi}\nabla_\mathbf{y}(e^\psi u).
\end{equation*}
In stationary convection-diffusion problems, $H(\textbf{curl})$- and $H(\textnormal{div})$-type fluxes can also be expressed in similar exponential forms \cite{wu2020simplex}.
Such symmetrized flux operators enable the application of the well-established analytical framework for purely diffusive problems (e.g., the Poisson equation) to convection-diffusion problems. This facilitates proofs of well-posedness and provides a foundation for designing stabilized numerical methods for convection-dominated regimes (e.g.,~\cite{adler2023stable,wu2020simplex}).

Our objective here is to generalize this exponentially-fitted flux operator to the spatiotemporal setting and to use it to reformulate the unified equation \eqref{eqn: physic_preserving_equation}.
We assume there exists a spatiotemporal potential function $\psi_0 = \psi_0(x,y,z,t)$ as a 0-form satisfying
\begin{equation*}
    \dzero\psi_0 = \mathbf{b}_1,
\end{equation*}
where $\mathbf{b}_1$ is the spatiotemporal convection 1-form introduced in \eqref{eqn: spatiotemporal_convection}.  
\begin{remark}
    The assumption $\mathrm{d}_0\psi_0 = \mathbf{b}_1$ may not hold in full
generality.  Indeed, a potential $\psi_0$ can exist only when the convection
field $\bm{\beta}$ is curl-free, because the de~Rham complex requires
every exact 1-form to be closed ($\mathrm{d}_1\mathrm{d}_0\psi_0 = 0$).  
Moreover, since the spacetime convection 1-form 
$\mathbf{b}_1=\alpha^{-1}\boldsymbol\beta\cdot \mathrm{d}\mathbf x-\varepsilon^{-1}\mathrm{d}t$
contains both spatial and temporal components, closedness additionally
requires $\alpha$ and $\bm{\beta}$ to be time-independent so that no
mixed space-time terms (such as $\mathrm{d}x\wedge\mathrm{d}t$) appear in $\mathrm d_1\mathbf b_1$.
Many practical applications (including fluid and drift-diffusion
models) feature irrotational convection fields. From a numerical viewpoint, local (spatial and temporal) approximations of the
convection field used in discretizations naturally yield locally exact
1-forms. The exponentially-fitted flux operator therefore remains well-defined even
when the physical convection field varies in time.
For the sake of simplicity in this section, we
retain the assumptions that $\nabla\times\bm{\beta}=0$ and
$\partial\alpha/\partial t=\partial\bm{\beta}/\partial t=\mathbf{0}$, so that the potential $\psi_0$ is globally well-defined.
\end{remark}
With this definition, the spatiotemporal convection-diffusion flux operator is expressed in exponential form as follows.
\begin{lemma}\label{lemma: exp_fitted_flux}
    For any $k$-form $\bm{\omega}_k$ and potential function $\psi_0$, the spatiotemporal flux operator satisfies
\begin{equation}\label{eqn: exponential_fitted_flux}
    \mathrm{J}_k\bm{\omega}_k = e^{-\psi_0}{\rm{d}}_k(e^{\psi_0}\bm{\omega}_k).
\end{equation}
\end{lemma}

\begin{proof}
    We invoke the graded product rule for the exterior derivative:
\begin{equation*}
    \mathrm{d}_{m+k}(\bm{\eta}_m\wedge\bm{\omega}_k) = \mathrm{d}_m\bm{\eta}_m\wedge\bm{\omega}_k+(-1)^m\bm{\eta}_m\wedge\mathrm{d}_k\bm{\omega}_k,
\end{equation*}
valid for any $m$-form $\bm{\eta}_m$ and $k$-form $\bm{\omega}_k$.
Specializing to the case $m=0$ with $\bm{\eta}_0 = e^{\psi_0}$, 
\begin{align*}
    e^{-\psi_0}{\rm{d}}_k(e^{\psi_0}\bm{\omega}_k) &= e^{-\psi_0}\left( \dzero e^{\psi_0}\wedge \bm{\omega}_k + e^{\psi_0} \mathrm{d}_k\bm{\omega}_k\right)\\
    & = e^{-\psi_0}\left( e^{\psi_0}\mathbf{b}_1\wedge \bm{\omega}_k + e^{\psi_0} {\rm{d}}_k\bm{\omega}_k\right)\\
    & = \mathbf{b}_1\wedge\bm{\omega}_k + {\rm{d}}_k\bm{\omega}_k,
\end{align*}
where we used the chain rule $\dzero e^{\psi_0} = e^{\psi_0} \dzero\psi_0 = e^{\psi_0} \mathbf{b}_1$.
This establishes \eqref{eqn: exponential_fitted_flux}.
\end{proof}

Using the result of \Cref{lemma: exp_fitted_flux}, the unified spatiotemporal equation \eqref{eqn: physic_preserving_equation} becomes
\begin{equation*}
    \delta^{k}_{1\alpha}{\rm{J}}_k\mathbf{u}_k + {\rm{d}}_{k-1}\delta^{k-1}_{\alpha1}\mathbf{u}_k=\delta^{k}_{1\alpha} e^{-\psi_0}\mathrm{d}_ke^{\psi_0}\mathbf{u}_k + {\rm{d}}_{k-1}\delta^{k-1}_{\alpha1}\mathbf{u}_k=\mathbf{f}_k.
\end{equation*}
This equation shows that the unified spatiotemporal formulation modifies the Hodge-Laplacian structure through the exponential weight $e^{\psi_0}$: 
\[
    \delta_{1\alpha}^k \mathrm{d}_k\ (\text{in Hodge Laplacian})
    \quad \Rightarrow \quad
    (\delta_{1\alpha}^k e^{-\psi_0}) (\mathrm{d}_k e^{\psi_0}) \ (\text{in convection-diffusion}).
\]

\begin{remark}
    Reformulating the unified equation in terms of the exponentially-fitted flux operator has several advantages. Namely, we demonstrate in \Cref{subsec: well-posedness} that this approach enables a rigorous well-posedness analysis. Moreover, it provides a natural framework for transferring structural properties and numerical methods developed for the Hodge Laplacian to convection-diffusion problems (see \cite{adler2023stable,wu2020simplex} for examples).
\end{remark}

\subsection{Variational formulation}\label{subsec: variational}

We now derive the variational formulation of the unified spatiotemporal convection-diffusion problem. Recall that the governing equation \eqref{eqn: physic_preserving_equation} is given by
\begin{equation*}
\left(\delta^{k}_{1\alpha}{\rm{J}}_k + {\rm{d}}_{k-1}\delta^{k-1}_{\alpha1}\right)\mathbf{u}_k=\mathbf{f}_k,
\end{equation*}
subject to the boundary conditions \eqref{eqn: xbc}--\eqref{eqn: ebc}:
\begin{equation*}
\left\{\begin{array}{ll}
   \mathbf{n}_{\mathbf{x}}\wedge \mathbf{u}_k = 0&\quad \text{on } \Gamma_{\mathbf{x}}, \\[0.5ex]
     \mathbf{n}_{t_0}\wedge \mathbf{u}_k = \mathbf{n}_{t_0}\wedge \mathbf{u}_{k,t_0}&\quad \text{on } \Gamma_{t_0},\\[0.5ex]
    \iota_{\mathbf{n}_{T}}(**_\alpha\mathrm{d}_k \mathbf{u}_k) = 0&\quad \text{on } \Gamma_{T}.
\end{array}\right.
\end{equation*}
Due to the Dirichlet-type conditions imposed on $\Gamma_{\mathbf{x}}$ and $\Gamma_{t_0}$, we define
\begin{equation*}
    H\Lambda_D^k(\Omega): = \left\{ \mathbf{v}_k\in H\Lambda^k(\Omega) \;:\;\mathbf{n}_\mathbf{x}\wedge\mathbf{v}_k=0\text{ on }\Gamma_\mathbf{x},\  \mathbf{n}_{t_0}\wedge\mathbf{v}_k=0 \text{ on }\Gamma_{t_0}\right\}.
\end{equation*}
For simplicity, we assume $\mathbf{u}_{k,t_0}=0$ and introduce the $L^2$ inner products
\begin{equation*}
(\mathbf{u}_k,\mathbf{v}_k)_\Omega
:=\int_\Omega\mathbf{u}_k\wedge *\mathbf{v}_k\quad\text{and}\quad \langle \mathbf{u}_k,\mathbf{v}_k \rangle_{\Gamma_T} := \int_{\Gamma_T} \textnormal{tr}\left(\mathbf{u}_k\wedge*\mathbf{v}_k\right).
\end{equation*}
We then define the bilinear form associated with \eqref{eqn: physic_preserving_equation} by
\begin{align*}
        \mathcal{B}(\mathbf{u}_k,\mathbf{v}_k)&=(-1)^{(k+1)(4-(k+1))}(**_\alpha\mathrm{J}_k\mathbf{u}_k,\mathrm{d}_k\mathbf{v}_k)_\Omega
        +\langle \mathbf{u}_k,\mathbf{v}_k \rangle_{\Gamma_T},
    \end{align*}
    for all $\mathbf{u}_k,\mathbf{v}_k \in H\Lambda_D^k(\Omega)$.
 We note that the sign factor $(-1)^{(k+1)(4-(k+1))}$, when combined with the double Hodge star operator $*\,*_\alpha$, reduces to a scalar coefficient (either $\alpha$ or $\varepsilon$) depending on whether the underlying form contains the temporal component $\mathrm{d}t$ (see \Cref{remark: Hodge_star} for details).  The variational formulation is now given in the following lemma.

\begin{lemma}\label{lemma: variational_problem} The unified equation \eqref{eqn: physic_preserving_equation} with boundary conditions \eqref{eqn: xbc}--\eqref{eqn: ebc} is equivalent to the variational problem: find $\mathbf{u}_k\in H\Lambda^k_D(\Omega)$ such that
\begin{equation}\label{eqn: weak_formulation}
    \mathcal{B}(\mathbf{u}_k,\mathbf{v}_k) = (\mathbf{f}_k,\mathbf{v}_k)_{\Omega},\quad\forall \mathbf{v}_k\in H\Lambda^k_D(\Omega).
\end{equation}
\end{lemma}

\begin{proof}
Let $\mathbf{v}_k\in H\Lambda_D^k(\Omega)$ be arbitrary. Testing \eqref{eqn: physic_preserving_equation} against $\mathbf{v}_k$ and integrating over $\Omega$, we obtain
\begin{equation*}
    (\delta_{1\alpha}^k\mathrm{J}_k\mathbf{u}_k,\mathbf{v}_k)_\Omega+(\mathrm{d}_{k-1}\delta_{\alpha1}^{k-1}\mathbf{u}_k,\mathbf{v}_k)_\Omega = \left(\mathbf{f}_k,\mathbf{v}_k\right)_\Omega.
\end{equation*}
For the first term, we use the relation $$\delta^{k}_{1\alpha} =(-1)^{(k+1)(4-(k+1))} (-1)*\left(\mathrm{d}_{4-(k+1)}\right)**\,*_\alpha=(-1)^{(k+1)(4-(k+1))}\delta^k**_\alpha$$ in four dimensions and apply Green's formula for differential forms \cite{mitrea2016hodge}:
    \begin{align*}
    (\delta_{1\alpha}^k\mathrm{J}_k\mathbf{u}_k,\mathbf{v}_k)_{\Omega}&=(-1)^{(k+1)(4-(k+1))}(\delta^k(**_\alpha\mathrm{J}_k\mathbf{u}_k),\mathbf{v}_k)_\Omega\\
    &=(-1)^{(k+1)(4-(k+1))}[\left(**_\alpha\mathrm{J}_k\mathbf{u}_k,\mathrm{d}_k\mathbf{v}_k\right)_\Omega-\langle\iota_{\mathbf{n}_{T}
    }(**_\alpha\mathrm{J}_k\mathbf{u}_k),\mathbf{v}_k\rangle_{\Gamma_T}].
\end{align*}
The Dirichlet-type boundary conditions prescribed in \eqref{eqn: xbc}--\eqref{eqn: tbc} eliminate the spatial and initial boundary terms via the wedge-contraction adjoint relation:
\begin{align*}
    &\langle\iota_{\mathbf{n}_{\mathbf{x}}
    }(**_\alpha\mathrm{J}_k\mathbf{u}_k),\mathbf{v}_k\rangle_{\Gamma_\mathbf{x}}=\langle**_\alpha\mathrm{J}_k\mathbf{u}_k,\mathbf{n}_{\mathbf{x}}\wedge\mathbf{v}_k\rangle_{\Gamma_\mathbf{x}}=0,\\
    &\langle\iota_{\mathbf{n}_{t_0}
    }(**_\alpha\mathrm{J}_k\mathbf{u}_k),\mathbf{v}_k\rangle_{\Gamma_{t_0}}=\langle**_\alpha\mathrm{J}_k\mathbf{u}_k,\mathbf{n}_{t_0}\wedge\mathbf{v}_k\rangle_{\Gamma_{t_0}}=0.
\end{align*}
Hence, only the terminal boundary term remains.  
Using \eqref{eqn: ebc}, we obtain
\begin{align*}
    & \quad \ \langle\iota_{\mathbf{n}_T}(**_\alpha\mathrm{J}_k\mathbf{u}_k),\mathbf{v}_k\rangle_{\Gamma_T} = \langle\iota_{\mathbf{n}_T}(**_\alpha(\mathrm{d}_k\mathbf{u}_k+\mathbf{b}_1\wedge\mathbf{u}_k)),\mathbf{v}_k\rangle_{\Gamma_T}\\
    &=\langle\iota_{\mathbf{n}_T}(**_\alpha(\mathbf{b}_1\wedge\mathbf{u}_k)),\mathbf{v}_k\rangle_{\Gamma_T}=(-1)^{(k+1)(4-(k+1))+1}\langle \mathbf{u}_k,\mathbf{v}_k\rangle_{\Gamma_T},
\end{align*}
where we have used the fact that, upon computing each $k$-form case,
\begin{equation*}
    \iota_{\mathbf{n}_T}(**_\alpha(\mathbf{b}_1\wedge\mathbf{u}_k)) = (-1)^{(k+1)(4-(k+1))+1}\mathbf{u}_k.
\end{equation*}
Collecting the above results, we obtain
\begin{equation*}
    (\delta_{1\alpha}^k\mathrm{J}_k\mathbf{u}_k,\mathbf{v}_k)_{\Omega} = (-1)^{(k+1)(4-(k+1))}(**_\alpha\mathrm{J}_k\mathbf{u}_k,\mathrm{d}_k\mathbf{v}_k)_\Omega + \langle \mathbf{u}_k,\mathbf{v}_k \rangle_{\Gamma_T}.
\end{equation*}

For the second term, the physics-preserving structure of \eqref{eqn: physic_preserving_equation} implies
\begin{equation*}
    \delta_{\alpha1}^{k-1}\mathbf{u}_k=0,
\end{equation*}
which corresponds to the divergence-, curl-, or gradient-free conditions depending on $k$. Hence, this term vanishes, completing the proof.
\end{proof}

\begin{remark}
    The variational problem \eqref{eqn: weak_formulation} exhibits a familiar form from standard convection-diffusion formulations, naturally induced by the physics-preserving structure.
However, the artificial temporal perturbation introduces an additional boundary condition \eqref{eqn: ebc} on $\Gamma_T$, producing an extra boundary contribution, $\langle \mathbf{u}_k,\mathbf{v}_k\rangle_{\Gamma_T}$, in the bilinear form $\mathcal{B}(\mathbf{u}_k,\mathbf{v}_k)$.
This boundary term does not pose any analytical difficulty, though, since by choosing $\mathbf{v}_k = e^{\psi_0}\mathbf{u}_k\not=0$, it contributes a strictly positive quantity (see \Cref{subsec: well-posedness} for details).
\end{remark}

\subsection{Well-posedness}\label{subsec: well-posedness}
To start, we define the solution space norm, 
\begin{equation*}
\|\mathbf{v}_k\|_{H\Lambda_D^k(\Omega)}^2: = \|\mathrm{d}_k\mathbf{v}_k\|_{L^2\Lambda^{k+1}(\Omega)}^2+ \|\mathbf{v}_k\|_{L^2\Lambda^{k}(\Omega)}^2 + \|\mathbf{v}_k\|_{L^2\Lambda^k(\Gamma_T)}^2,
\end{equation*}
where
\begin{equation*}
\|\mathbf{v}_k\|_{L^2\Lambda^{k}(\Omega)}^2: = (\mathbf{v}_k,\mathbf{v}_k)_\Omega,\quad\|\mathbf{v}_k\|_{L^2\Lambda^{k}(\Gamma_T)}^2: = \langle\mathbf{v}_k,\mathbf{v}_k\rangle_{\Gamma_T}.
\end{equation*}
We first establish a norm equivalence between $\|\mathbf{v}_k\|_{H\Lambda_D^k(\Omega)}$ and the exponentially-weighted norm, $\|e^{\psi_0}\mathbf{v}_k\|_{H\Lambda_D^k(\Omega)}$.
\begin{lemma}\label{lemma: norm_equiv}
    For any $\mathbf{v}_k\in H\Lambda_D^k(\Omega)$, there exist positive constants $\gamma_*$ and $\gamma^*$, depending only on the artificial parameter $\varepsilon$, the spatial diffusion coefficient $\alpha$, and the convection field $\bm{\beta}$, such that
    \begin{equation*}
      \gamma_*\|e^{\psi_0}\mathbf{v}_k\|_{H\Lambda_D^k(\Omega)}\leq \|\mathbf{v}_k\|_{H\Lambda_D^k(\Omega)}\leq \gamma^* \|e^{\psi_0}\mathbf{v}_k\|_{H\Lambda_D^k(\Omega)}.
    \end{equation*}
\end{lemma}
\begin{proof}
Starting with the upper bound,
note that $\|\mathbf{v}_k\|_{H\Lambda_D^k(\Omega)}$ consists of an exterior derivative term, an $L^2$ mass term, and a boundary term.
Considering the exterior derivative term, we apply the graded product rule and the triangle inequality:
    \begin{align*}
\|\mathrm{d}_k\mathbf{v}_k\|_{L^2\Lambda^{k+1}(\Omega)} \leq \|\mathrm{d}_0e^{-\psi_0}\wedge e^{\psi_0}\mathbf{v}_k\|_{L^2\Lambda^{k+1}(\Omega)}+\|e^{-\psi_0}\wedge\mathrm{d}_k(e^{\psi_0}\mathbf{v}_k)\|_{L^2\Lambda^{k+1}(\Omega)}.
    \end{align*}
For the first term, using $\mathrm{d}_0e^{-\psi_0} = e^{-\psi_0}\mathbf{b}_1$,
    \begin{align*}
        \|\mathrm{d}_0e^{-\psi_0}\wedge e^{\psi_0}\mathbf{v}_k\|_{L^2\Lambda^{k+1}(\Omega)}&=\|e^{-\psi_0}\mathbf{b}_1\wedge e^{\psi_0}\mathbf{v}_k\|_{L^2\Lambda^{k+1}(\Omega)}\\
        &\leq C_{\mathbf{b}_1}C_{\psi_0}\| e^{\psi_0}\mathbf{v}_k\|_{L^2\Lambda^{k}(\Omega)}\\
        &\leq C_{\mathbf{b}_1}C_{\psi_0}\| e^{\psi_0}\mathbf{v}_k\|_{H\Lambda_D^k(\Omega)}.
    \end{align*}
    Here, $C_{\mathbf{b}_1}=\|\mathbf{b}_1\|_{L^2\Lambda^1(\Omega)}$,
    and $C_{\psi_0}=\|e^{-\psi_0}\|_{L^\infty\Lambda^0(\Omega)}$, where
$L^\infty\Lambda^k(\Omega)$ denotes the space of essentially bounded k-forms.
The second term is straightforwardly bounded using the boundedness of $e^{-\psi_0}$:
    \begin{align*}
        \|e^{-\psi_0}\wedge\mathrm{d}_k(e^{\psi_0}\mathbf{v}_k)\|_{L^2\Lambda^{k+1}(\Omega)}&\leq C_{\psi_0}\|\mathrm{d}_k(e^{\psi_0}\mathbf{v}_k)\|_{L^2\Lambda^{k+1}(\Omega)}\\&\leq C_{\psi_0}\|e^{\psi_0}\mathbf{v}_k\|_{H\Lambda_D^k(\Omega)}.
    \end{align*}
Combining the bounds on the two terms, 
    \begin{equation*}
        \|\mathrm{d}_k\mathbf{v}_k\|_{L^2\Lambda^{k+1}(\Omega)}\leq \left(C_{\mathbf{b}_1}+1\right)C_{\psi_0}\|e^{\psi_0}\mathbf{v}_k\|_{H\Lambda_D^k(\Omega)}.
    \end{equation*}
Similarly, for the $L^2$ mass and boundary terms, the boundedness of $e^{-\psi_0}$ directly yields
    \begin{align*}
        &\|\mathbf{v}_k\|_{L^2\Lambda^k(\Omega)}\leq C_{\psi_0}\|e^{\psi_0}\mathbf{v}_k\|_{L^2\Lambda^k(\Omega)}\leq C_{\psi_0}\|e^{\psi_0}\mathbf{v}_k\|_{H\Lambda_D^k(\Omega)},\\
        &\|\mathbf{v}_k\|_{L^2\Lambda^k(\Gamma_T)}\leq C_{\psi_0}\|e^{\psi_0}\mathbf{v}_k\|_{L^2\Lambda^k(\Gamma_T)}\leq C_{\psi_0}\|e^{\psi_0}\mathbf{v}_k\|_{H\Lambda_D^k(\Omega)}.
    \end{align*}
Collecting the three terms establishes the upper bound with $\gamma^*:=\left(C_{\mathbf{b}_1}+3\right)C_{\psi_0}$.

    The lower bound follows similarly by applying the product rule and using the boundedness of $e^{\psi_0}$:
\begin{align*}
    \|\mathrm{d}_k(e^{\psi_0}\mathbf{v}_k)\|_{L^2\Lambda^{k+1}(\Omega)}&\leq \|\mathrm{d}_0 e^{\psi_0}\wedge\mathbf{v}_k\|_{L^2\Lambda^{k+1}(\Omega)}+\|e^{\psi_0}\wedge\mathrm{d}_k\mathbf{v}_k\|_{L^2\Lambda^{k+1}(\Omega)}\\
    &\leq\|e^{\psi_0}\mathbf{b}_1\wedge\mathbf{v}_k\|_{L^2\Lambda^{k+1}(\Omega)}+C_{\psi_0}^\dagger\|\mathrm{d}_k\mathbf{v}_k\|_{L^2\Lambda^{k+1}(\Omega)}\\
    &\leq C_{\mathbf{b}_1}C_{\psi_0}^\dagger\|\mathbf{v}_k\|_{L^2\Lambda^{k}(\Omega)}  +C_{\psi_0}^\dagger\|\mathrm{d}_k\mathbf{v}_k\|_{L^2\Lambda^{k+1}(\Omega)}\\
    &\leq\left(C_{\mathbf{b}_1}+1\right)C_{\psi_0}^\dagger\|\mathbf{v}_k\|_{H\Lambda_D^k(\Omega)},
\end{align*}
where $C_{\psi_0}^\dagger=\|e^{\psi_0}\|_{L^\infty\Lambda^0(\Omega)}$. Moreover,
\begin{align*}
    \|e^{\psi_0}\mathbf{v}_k\|_{L^2\Lambda^k(\Omega)}\leq C_{\psi_0}^\dagger\|\mathbf{v}_k\|_{H\Lambda_D^k(\Omega)}, \quad \|e^{\psi_0}\mathbf{v}_k\|_{L^2\Lambda^k(\Gamma_T)}\leq C_{\psi_0}^\dagger\|\mathbf{v}_k\|_{H\Lambda_D^k(\Omega)}.
\end{align*}
This completes the proof of the lower bound, with $\gamma_*:=\left[\left(C_{\mathbf{b}_1}+3\right)C_{\psi_0}^\dagger\right]^{-1}$.
\end{proof}

Next, using this norm equivalence, we establish an inf-sup condition and continuity of the bilinear form.
\begin{lemma}[Inf-sup condition]
    There exists a constant $c_{\varepsilon,\alpha,\bm{\beta}}>0$, depending only on the artificial parameter $\varepsilon$, the spatial diffusion coefficient $\alpha$, and the convection field $\bm{\beta}$, such that
    \begin{equation}\label{eqn: corecivity}
        \inf_{\mathbf{u}_k\in H\Lambda_D^k(\Omega)}\sup_{\mathbf{v}_k\in H\Lambda_D^k(\Omega)}\frac{\mathcal{B}(\mathbf{u}_k,\mathbf{v}_k)}{\|\mathbf{u}_k\|_{H\Lambda_D^k(\Omega)}\|\mathbf{v}_k\|_{H\Lambda_D^k(\Omega)}}\geq c_{\varepsilon,\alpha,\bm{\beta}}.
     \end{equation}
\end{lemma}
\begin{proof}
We begin by utilizing the result from \Cref{lemma: variational_problem} and the definition of the exponentially-fitted flux \eqref{eqn: exponential_fitted_flux} to rewrite the bilinear form as
\begin{align*}
    \mathcal{B}(\mathbf{u}_k,\mathbf{v}_k)&=(-1)^{(k+1)(4-(k+1))}(**_\alpha\mathrm{J}_k\mathbf{u}_k,\mathrm{d}_k\mathbf{v}_k)_\Omega + \langle\mathbf{u}_k,\mathbf{v}_k\rangle_{\Gamma_T}\\
    &=(-1)^{(k+1)(4-(k+1))}(**_\alpha e^{-\psi_0}\mathrm{d}_k(e^{\psi_0}\mathbf{u}_k),\mathrm{d}_k\mathbf{v}_k)_\Omega + \langle\mathbf{u}_k,\mathbf{v}_k\rangle_{\Gamma_T}.
\end{align*}
Choosing the test function \(\mathbf{v}_k = e^{\psi_0}\mathbf{u}_k\not=0\) and using the property of the scaled double Hodge star operator in \Cref{remark: Hodge_star},
we obtain 
\begin{equation*}
    \mathcal{B}(\mathbf{u}_k,e^{\psi_0}\mathbf{u}_k)\geq \frac{\min(\varepsilon,\alpha)}{C_{\psi_0}^\dagger}\left(\|\mathrm{d}_k(e^{\psi_0}\mathbf{u}_k)\|_{L^2\Lambda^{k+1}(\Omega)}^2 + \|e^{\psi_0}\mathbf{u}_k\|_{L^2\Lambda^k(\Gamma_T)}^2\right).
\end{equation*}
Using the definition of the solution space norm, the Poincar\'e-Friedrichs inequality \cite{arnold2018finite}, and the norm equivalence in \Cref{lemma: norm_equiv}, we obtain the estimate
\begin{equation*}
    \mathcal{B}(\mathbf{u}_k,e^{\psi_0}\mathbf{u}_k)\geq C\|e^{\psi_0}\mathbf{u}_k\|_{H\Lambda_D^k(\Omega)}^2\geq (C/\gamma^*)\|\mathbf{u}_k\|_{H\Lambda_D^k(\Omega)}\|e^{\psi_0}\mathbf{u}_k\|_{H\Lambda_D^k(\Omega)}.
\end{equation*}
This implies the stated inf-sup condition \eqref{eqn: corecivity}
with $c_{\varepsilon,\alpha,\bm{\beta}}=C/\gamma^*$.
\end{proof}

\begin{lemma}[Continuity]
    There exists a constant $C_{\varepsilon,\alpha,\bm{\beta}}>0$, depending only on $\varepsilon$, $\alpha$, and $\bm{\beta}$, such that
    \begin{equation}\label{eqn: continuity}
        |\mathcal{B}(\mathbf{u}_k,\mathbf{v}_k)|\leq C_{\varepsilon,\alpha,\bm{\beta}}\|\mathbf{u}_k\|_{H\Lambda_D^k(\Omega)}\|\mathbf{v}_k\|_{H\Lambda_D^k(\Omega)}.
    \end{equation}
\end{lemma}
\begin{proof}
The continuity condition can be shown using the Cauchy-Schwarz inequality. 
Indeed, for all $\mathbf{u}_k,\mathbf{v}_k\in H\Lambda_D^k(\Omega)$, we have
\begin{align*}
    \left|\mathcal{B}(\mathbf{u}_k,\mathbf{v}_k)\right|&\leq C_{\alpha}( \left|\left(\mathrm{d}_k\mathbf{u}_k,\mathrm{d}_k\mathbf{v}_k\right)_\Omega\right| + \left|\left(\mathbf{b}_1\wedge\mathbf{u}_k,\mathrm{d}_k\mathbf{v}_k\right)_\Omega\right| + \left|\langle\mathbf{u}_k,\mathbf{v}_k\rangle_{\Gamma_T}\right|)\\
    &\leq C_{\varepsilon,\alpha,\bm{\beta}}\|\mathbf{u}_k\|_{H\Lambda_D^k(\Omega)}\|\mathbf{v}_k\|_{H\Lambda_D^k(\Omega)},
\end{align*}
where $C_{\alpha} = \max(\alpha,1)$ and $C_{\varepsilon,\alpha,\bm{\beta}}=C_\alpha(\max(C_{\mathbf{b}_1},1))$.    
\end{proof}

The established inf-sup condition together with the continuity of the bilinear
form $\mathcal{B}(\cdot,\cdot)$ implies the well-posedness of the variational
problem \eqref{eqn: weak_formulation} via the Banach-Ne\v{c}as-Babu\v{s}ka
theorem.
\begin{theorem}[Well-posedness]
Assume that $\mathbf{f}_k\in L^2\Lambda^k(\Omega)$. Suppose that the bilinear form $\mathcal{B}(\cdot,\cdot)$ is continuous \eqref{eqn: continuity} and satisfies an inf-sup condition \eqref{eqn: corecivity} on $H\Lambda^k_D(\Omega)$.
Then, there exists a unique solution
$\mathbf{u}_k \in H\Lambda_D^k(\Omega)$ to
problem \eqref{eqn: weak_formulation}, and that solution satisfies
\begin{equation*}
\|\mathbf{u}_k\|_{H\Lambda_D^k(\Omega)}
\leq C \|\mathbf{f}_k\|_{L^2\Lambda^k(\Omega)},
\end{equation*}
where the constant $C>0$ is independent of $\mathbf{f}_k$.
\end{theorem}
\begin{proof}
Since $\mathbf{f}_k \in L^2\Lambda^k(\Omega)$, the linear functional
$\mathbf{v}_k \mapsto (\mathbf{f}_k,\mathbf{v}_k)_\Omega$
is bounded on $H\Lambda_D^k(\Omega)$.
The result follows directly from the Banach-Ne\v{c}as-Babu\v{s}ka theorem (see e.g., \cite{ern2004theory}).
\end{proof}

\section{Convergence analysis}\label{sec: convergence}

Let $\mathbf{u}_k^\varepsilon$ be the solution of the spatiotemporal problem~\eqref{eqn: physic_preserving_equation}, which includes an artificial temporal perturbation $\varepsilon>0$, and let $\mathbf{u}_k^0$ be the solution of the original convection-diffusion problems~\eqref{sys: conv_diff_grad}--\eqref{sys: conv_diff_temporal}. 
The goal of this section is to establish the convergence of the perturbed solution $\mathbf{u}_k^\varepsilon$ to the limit solution $\mathbf{u}_k^0$ as $\varepsilon \to 0$. We drop the subscript $k$ in this section for the sake of simplicity.

We denote the spatial Sobolev space associated with form degree $k$ as $V_\mathbf{x}$ (e.g., $V_\mathbf{x} = H_0^1(\Omega_\mathbf{x})$ for $k=0$, and $V_\mathbf{x}=H_0(\textbf{curl};\Omega_\mathbf{x})$ for $k=1$), so that $$\|\cdot\|_{L^2(\Omega_\mathbf{x})}\leq \|\cdot\|_{V_\mathbf{x}} = \left(\|\cdot\|_{L^2(\Omega_\mathbf{x})}^2 + |\cdot|_{V_\mathbf{x}}^2\right)^{1/2},$$
where $|\cdot|_{V_\mathbf{x}}$ denotes the corresponding seminorm (e.g., the $H^1$ or $H(\mathbf{curl})$ seminorm).
The operator $\mathcal{L}$ represents the spatial differential operator containing both diffusion and convection terms, for instance, $\mathcal{L}(\cdot) = \nabla\!\cdot(\alpha\nabla(\cdot) + \bm{\beta}(\cdot))$ when $k=0$.
For a Banach space $V$ and $1\leq p\leq\infty$, we denote by $L^p([t_0,T];V)$
the Bochner space of strongly measurable functions \cite{adams2003sobolev,evans2022partial}.
That is, if $v\in L^p([t_0,T];V)$, then $v(t)\in V$ for almost every 
$t\in[t_0,T]$, and $\|v(t)\|_V$ is $L^p$-integrable in time.

\begin{theorem}\label{thm: convergence}
Let $\mathbf{e} = \mathbf{u}^\varepsilon - \mathbf{u}^0$ be the error function.
Assume the regularity 
$\mathbf{u}^0,\mathbf{u}^\varepsilon \in L^\infty([t_0,T];V_\mathbf{x})$, $\mathbf{u}_t^0,\mathbf{u}_t^\varepsilon \in L^\infty([t_0,T];V_\mathbf{x})$, and $\mathbf{u}_{tt}^0,\mathbf{u}_{tt}^\varepsilon\in L^2([t_0,T];L^2(\Omega_\mathbf{x}))$.
Then, for sufficiently small $\varepsilon$, the following energy estimate holds for all $t\in [t_0,T]$: 
\begin{align*}
    &\frac{1}{2}\|\mathbf{e}(t)\|_{L^2(\Omega_\mathbf{x})}^2+c_0\int_{t_0}^t\|\mathbf{e}(s)\|_{V_\mathbf{x}}^2\,\mathrm{d}s\\
    &\qquad\qquad\leq\frac{\varepsilon}{2}e^{2C_{\ell}(t-t_0)}\int_{t_0}^t\|\mathbf{u}_{tt}^0(s)\|_{L^2(\Omega_\mathbf{x})}^2\,\mathrm{d}s + \varepsilon\|\mathbf{e}_t(t)\|_{L^2(\Omega_\mathbf{x})}\|\mathbf{e}(t)\|_{V_\mathbf{x}}.
\end{align*}
Here, $C_{\ell}=\|\bm{\beta}\|_{L^\infty(\Omega_\mathbf{x})}/(2\alpha_0)>0$, and $c_0=\alpha_0/2-\varepsilon((C_{\ell})^2+1/2)$.
\end{theorem}
\begin{proof}
    Subtracting the equations satisfied by $\mathbf{u}^\varepsilon$ and $\mathbf{u}^0$ yields the error equation
\begin{equation*}
    -\varepsilon \mathbf{e}_{tt} + \mathbf{e}_t - \mathcal{L}\mathbf{e} = \varepsilon \mathbf{u}_{tt}^0.
\end{equation*}
We test this equation with $\mathbf{e}$ and integrate over the spatial domain $\Omega_\mathbf{x}$.
The first term is written as
\begin{equation*}
    -\varepsilon(\mathbf{e}_{tt}(t),\mathbf{e}(t))_{\Omega_\mathbf{x}} = -\varepsilon\frac{\mathrm{d}}{\mathrm{d}t}(\mathbf{e}_t(t),\mathbf{e}(t))_{\Omega_\mathbf{x}} + \varepsilon \|\mathbf{e}_t(t)\|_{L^2(\Omega_\mathbf{x})}^2.
\end{equation*}
By the same argument,
\begin{equation*}
    (\mathbf{e}_t(t),\mathbf{e}(t))_{\Omega_\mathbf{x}} = \frac{1}{2}\frac{\mathrm{d}}{\mathrm{d}t}\|\mathbf{e}(t)\|_{L^2(\Omega_\mathbf{x})}^2.
\end{equation*}
For the spatial convection-diffusion operator, integration by parts
with homogeneous Dirichlet-type boundary conditions (e.g., \eqref{eqn: grad_boundary}, \eqref{eqn: curl_boundary}, and \eqref{eqn: div_boundary}) gives
\begin{equation*}
    -(\mathcal{L}\mathbf{e}(t),\mathbf{e}(t))_{\Omega_\mathbf{x}}\geq \frac{\alpha_0}{2}\|\mathbf{e}(t)\|_{V_\mathbf{x}}^2 - C_{\ell}\|\mathbf{e}(t)\|_{L^2(\Omega_\mathbf{x})}^2,
\end{equation*}
where $\alpha \ge \alpha_0 > 0$ is the ellipticity constant of the diffusion coefficient,
and $C_{\ell} = \|\bm{\beta}\|_{L^\infty(\Omega_\mathbf{x})}/(2\alpha_0)>0$.
By combining the above results and applying the Cauchy-Schwarz inequality, we obtain
\begin{align*}
    &\frac{1}{2}\frac{\mathrm{d}}{\mathrm{d}t}\|\mathbf{e}(t)\|_{L^2(\Omega_\mathbf{x})}^2 -\varepsilon\frac{\mathrm{d}}{\mathrm{d}t}(\mathbf{e}_t(t),\mathbf{e}(t))_{\Omega_\mathbf{x}} + \frac{\alpha_0}{2}\|\mathbf{e}(t)\|_{V_\mathbf{x}}^2 + \varepsilon \|\mathbf{e}_t(t)\|_{L^2(\Omega_\mathbf{x})}^2\\
    &\qquad\qquad\leq C_{\ell}\|\mathbf{e}(t)\|_{L^2(\Omega_\mathbf{x})}^2 + \varepsilon\|\mathbf{u}_{tt}^0(t)\|_{L^2(\Omega_\mathbf{x})}\|\mathbf{e}(t)\|_{V_\mathbf{x}} \\
    &\qquad\qquad\qquad\qquad- 2\varepsilon C_{\ell}(\mathbf{e}_t(t),\mathbf{e}(t))_{\Omega_\mathbf{x}} + 2\varepsilon C_{\ell}\|\mathbf{e}_t(t)\|_{L^2(\Omega_\mathbf{x})}\|\mathbf{e}(t)\|_{V_\mathbf{x}}.
\end{align*}
We introduce the modified energy functional,
\begin{equation*}
    \Phi(t) = \frac{1}{2}\|\mathbf{e}(t)\|_{L^2(\Omega_\mathbf{x})}^2 - \varepsilon(\mathbf{e}_t(t),\mathbf{e}(t))_{\Omega_\mathbf{x}},
\end{equation*}
which allows the above inequality to be rewritten as
\begin{align*}
    \frac{\mathrm{d}}{\mathrm{d}t}\Phi(t) + c_0\|\mathbf{e}(t)\|_{V_\mathbf{x}}^2
    \leq 2C_{\ell}\Phi(t) + \frac{\varepsilon}{2}\|\mathbf{u}_{tt}^0(t)\|_{L^2(\Omega_\mathbf{x})}^2,
\end{align*}
where $$c_0 = \frac{\alpha_0}{2} - \varepsilon \left((C_{\ell})^2+\frac{1}{2}\right) > 0,$$ for sufficiently small $\varepsilon$.
Applying Gr\"{o}nwall’s inequality and noting that $\Phi(t_0)=0$ (since $\mathbf{e}(\mathbf{x},t_0)=\mathbf{0}$ in $\Omega_\mathbf{x}$), we obtain
\begin{align*}
    \Phi(t) + c_0\int_{t_0}^t\|\mathbf{e}(s)\|_{V_\mathbf{x}}^2\,\mathrm{d}s
    \leq \frac{\varepsilon}{2}e^{2C_{\ell}(t-t_0)}\int_{t_0}^t\|\mathbf{u}_{tt}^0(s)\|_{L^2(\Omega_\mathbf{x})}^2\,\mathrm{d}s.
\end{align*}
Finally, substituting the definition of $\Phi(t)$ and using the Cauchy-Schwarz inequality for the mixed term, we conclude that for all $t\in[t_0,T]$,
\begin{align*}
    &\frac{1}{2}\|\mathbf{e}(t)\|_{L^2(\Omega_\mathbf{x})}^2+c_0\int_{t_0}^t\|\mathbf{e}(s)\|_{V_\mathbf{x}}^2\,\mathrm{d}s\\
    &\qquad\qquad\leq\frac{\varepsilon}{2}e^{2C_{\ell}(t-t_0)}\int_{t_0}^t\|\mathbf{u}_{tt}^0(s)\|_{L^2(\Omega_\mathbf{x})}^2\,\mathrm{d}s + \varepsilon\|\mathbf{e}_t(t)\|_{L^2(\Omega_\mathbf{x})}\|\mathbf{e}(t)\|_{V_\mathbf{x}}.
\end{align*}
\end{proof}

Therefore, \Cref{thm: convergence} shows that $\mathbf{e}(t)\to\mathbf{0}$ in $L^2(\Omega_\mathbf{x})$ for all $t\in[t_0,T]$, and that $\mathbf{e}\to\mathbf{0}$ in $L^2([t_0,T];V_\mathbf{x})$, since $c_0\to\alpha_0/2$, ensuring that $\mathbf{u}^\varepsilon\to\mathbf{u}^0$ as $\varepsilon\to0$.

\section{Conclusions}
\label{sec: conclusions}
In this work, we developed a unified 4D spatiotemporal formulation for time-dependent convection-diffusion problems using differential forms and exterior calculus. Extending the space-time framework of \cite{bank2017arbitrary}, we generalized the formulation beyond the $H(\textbf{grad})$ setting to encompass the full family of problems posed in $H(\textbf{grad})$, $H(\textbf{curl})$, $H(\textnormal{div})$, and the purely temporal $L^2$ framework. This required the introduction of appropriate spatiotemporal diffusion tensors, convection forms, and unified boundary conditions, together with a small artificial temporal perturbation ensuring the nondegeneracy of the diffusion operator.
Importantly, the resulting 4D governing equation naturally embeds physics-preserving structures, such as divergence-free and curl-free conditions, without imposing them externally.
We also introduced an exponentially-fitted spatiotemporal flux operator that symmetrizes the convection-diffusion flux and mirrors structural features of a Hodge Laplacian. This operator supports a unified variational formulation and facilitates the analysis of well-posedness across all form degrees. Finally, we established that the perturbed spatiotemporal formulation converges to the original convection-diffusion problems as the temporal perturbation parameter tends to zero, thereby ensuring consistency.

Future work will investigate variants of the unified spatiotemporal formulation that incorporate Lie convection operators, develop 4D monotone discretizations for convection-dominated regimes, and extend robust solver strategies, such as those in \cite{wang2025robust}, to the spatiotemporal setting.



\section*{Acknowledgments}
The authors thank SeongHee Jeong and Ludmil Zikatanov for their valuable comments, suggestions, and insights toward the manuscript.
\bibliographystyle{siamplain}
\bibliography{references}
\end{document}